\documentclass{siamltex}

\pdfoutput=1
\usepackage[T1]{fontenc}
\usepackage[latin1]{inputenc}
\usepackage{fancybox}

\usepackage{amssymb}
\usepackage{amsmath}
\usepackage{mathpazo}

\usepackage{xcolor}
\usepackage{color}
\usepackage{float}
\usepackage[final]{graphicx}
\usepackage{tikz}
\usetikzlibrary{decorations}
\usetikzlibrary{patterns}

\usepackage{enumerate}
\usepackage[%
  colorlinks%
  ,bookmarks={false}%
  ,pdftitle={paper1}%
  ,pdfsubject={On the Dirichlet Problem for First Order Linear Hyperbolic PDEs on Bounded Domains with Mere Inflow Boundary}%
  ,pdfkeywords={}%
  ,pdfauthor={Thomas März <maerzt@ma.tum.de>}%
  ,plainpages=false
  ]{hyperref}

\input{macros}

\title{On the Dirichlet Problem for First Order Linear Hyperbolic PDEs on Bounded Domains with Mere Inflow Boundary}

\author{Thomas M\"{a}rz\thanks{Zentrum Mathematik,
        Technische Universit\"{a}t M\"{u}nchen,
        Boltzmannstr. 3, 85747 Garching, Germany
        ({\tt maerzt@ma.tum.de}). Manuscript as of \today.}}

\begin{document}

\maketitle

% On the Dirichlet Problem for First Order Linear Hyperbolic PDEs on Bounded Domains with Mere Inflow Boundary
 
\begin{abstract} 
Here we study the Dirichlet problem for first order linear and quasi-linear hyperbolic PDEs on a simply connected bounded domain of $\R^2$, where the domain has an
interior outflow set and a mere inflow boundary. By means of a Lyapunov function we show the existence of a unique solution in the 
space of functions of bounded variation
and its continuous dependence on all the data of the linear problem.
Finally, we conclude the existence of a solution to the quasi-linear case by utilizing the Schauder fixed point theorem.
This type of problems considered here appears in applications such as transport based image inpainting. 
\end{abstract}



\section{Introduction}\label{Sect:Intro}
The subject of this paper is hyperbolic linear and quasi-linear scalar PDEs of first order in two space dimensions.
We consider the Dirichlet problem on simply connected bounded domains $\Omega$ and bounded functions $u_0$ of bounded variation as boundary data.
The linear PDE is stated in the space $\BV$ -- the functions of bounded variation -- and formulated as
\begin{align}\label{eqn:LinPDE}
	\SP< {c(x)}, {Du} > &= f(x) \cdot \Lm^2 \quad \mbox{in} \quad \Omega \wo \Sigma\;,  & u|_{\bd \Omega} &= u_0 \;.
\end{align}
Here, $\SP<\dotarg,\dotarg >$ denotes the scalar product in $\R^2$, $Du$ is the derivative measure of $u \in \BV$ and $\Lm^2$ the Lebesgue measure on $\R^2$.

Because the PDE \refEq{eqn:LinPDE} is hyperbolic and our aim is to ensure global existence and uniqueness, we have to rule out the case of characteristic points
and therefor need an additional requirement.
Our additional requirement is the existence of a global Lyapunov function $T:\Omega \to \R$ for the transport field $c$.
That is $c$ is assumed to satisfy the Lyapunov or causality condition
\begin{align}\label{eqn:Causal}
	\SP< {c(x)}, {\nabla T(x)}> \geq \beta \cdot |\nabla T(x)| \;, & \qquad \beta > 0 \;.
\end{align}
Besides that, the Lyapunov function $T$ is zero on the boundary and increases towards its maximal level $\Sigma$.

Since the boundary is a level line of $T$, condition \refEq{eqn:Causal} means in particular that every point of the boundary is an inflow point.
On the other hand PDE \refEq{eqn:LinPDE} is an advection equation, thus we need an outflow set. For our case the outflow set is in the interior of the domain
and is given by $\Sigma$. The outflow sets considered here are connected tree-like sets compactly contained in $\Omega$.

Assuming $c \in \C^1(\Omega \wo \Sigma)^2$, equation \refEq{eqn:LinPDE} is a well-behaved advection equation and thus we will
employ the method of characteristics to construct the solution.
A related problem, the Cauchy problem for quasi-linear equations in conservative form
\begin{align}\label{eqn:Conserve}
	\pd_t u(t,x) + \pd_x a(t,x,u(t,x)) &= 0 \;, & u(0,x) = u_0(x) \;,
\end{align}
with bounded data $u_0$ of locally of bounded variation, has already been studied, see e.g. \cite{Volpert67} and \cite{ConwaySmoller66}.
Here, the non-linearity of $a(t,x,u)$ w.r.t. the variable $u$ causes the solution to develop shocks within finite time, which in turn makes a notion of generalized solution necessary.
In \cite{Volpert67}, the author proves the existence and uniqueness, under some suitable assumptions on the jump discontinuities, in $\BV \cap \L^{\infty}$ employing the method
of vanishing viscosity.

Our problem is slightly different. The set $\Sigma$ is globally attractive by the causality condition \refEq{eqn:Causal} and thus characteristics will meet there.
Because equation \refEq{eqn:LinPDE} is a well-behaved advection, the solution cannot develop shocks before the characteristics reach the set $\Sigma$
but we will in general obtain shocks on $\Sigma$. This is the case even for smooth boundary data.

In this paper we will show existence, uniqueness and stability of the solution to problem \refEq{eqn:LinPDE} in the space $\BV(\Omega) \cap L^{\infty}(\Omega)$.

An important feature of our solution is that it is defined on all of $\Omega$. That means, we obtain a description of the shocks at $\Sigma$ in the $\BV$-framework
and so, in this sense, we can close the gap:  $u$ solves problem \refEq{eqn:LinPDE} in $\Omega \wo \Sigma$, but $u \in \BV(\Omega)$ and not only $u \in \BV(\Omega\wo\Sigma)$.

The benefit of ''closing the gap'' is the good functional analytical properties of $\BV(\Omega)$, in particular compactness, which we will exploit to tackle the quasi-linear case: 
\begin{align}\label{eqn:QuasiLinPDE}
	\SP< {c[u](x)\,} , {D u} > &= f[u](x) \cdot \Lm^2  \qquad \mbox{in} \quad \Omega \wo \Sigma \;, \qquad & u|_{\bd \Omega} &= u_0 \;.
\end{align}
In contrast to equation \refEq{eqn:Conserve}, we consider quasi-linear problems in non-conservative form but we allow the dependence on $u$ of the transport field $c$ 
and the right-hand side $f$ to be of a general functional type.
This means that they depend not only on the value $u(x)$ but on the whole function, thus the coefficients of the PDE \refEq{eqn:QuasiLinPDE} are maps 
\begin{equation}
	\begin{aligned}
		f: \FX \to \FR_1 &\; , \; u \to f[u] \;, & \qquad  f&[\dotarg](x): \FX \to \R^{~}  \; , \; u \to f[u](x) \;, \\
		c: \FX \to \FR_2 & \; , \; u \to c[u] \;, & \qquad c&[\dotarg](x): \FX \to \R^2  \;, \; u \to c[u](x) \;, 
	\end{aligned}
\end{equation}
with $\FX \subset \BV(\Omega)$ and $\FR_1$, $\FR_2$ being subsets of suitable function spaces defined on $\Omega \wo \Sigma$.

We will approach the quasi-linear problem \refEq{eqn:QuasiLinPDE} by using the theory of the linear problem. 
Fixing the functional argument of the coefficients by some $v \in \FX$, we obtain the linear PDE
\begin{align}
	\SP< {c[v](x) \,},{D u} > &= f[v](x) \cdot \Lm^2  \qquad \mbox{in} \quad \Omega \wo \Sigma \;, \qquad &
	u|_{\bd \Omega} &= u_0 \;,
\end{align}
with a solution $U[v] \in \FX$ depending on $v$. If now the operator $U$ admits a fixed point $u = U[u]$, we obtain a solution to problem \refEq{eqn:QuasiLinPDE}.

The compactness mentioned above and the stability properties of the linear case allow us to establish the existence of a solution $u = U[u]$ by employing the Schauder fixed point theorem.

An application of the quasi-linear equation \refEq{eqn:QuasiLinPDE} is transport based image inpainting. 
The term image inpainting means the retouching of undesired or damaged portions of a digital image.
Telea suggested in \cite{Telea04} a fast algorithm for digital image inpainting which is however not adapted to the image under consideration. 

In \cite{FBTM07}, we have analyzed Telea's algorithm in a continuous setting and have seen that it is consistent with the following model
\begin{align*}
	\SP< {\nabla T(x)}, {\nabla u} > &= 0  \quad \mbox{in} \quad \Omega \wo \Sigma\;,  & u|_{\bd \Omega} &= u_0 \;.
\end{align*}
Moreover, we have generalized this linear model to problem \refEq{eqn:LinPDE} with side condition \refEq{eqn:Causal} and finally we adapted the coefficients of the PDE
to the image which resulted in a quasi-linear model of type \refEq{eqn:QuasiLinPDE}.
With our image adaption we obtained an image inpainting method which produces results of high visual quality and performs almost as fast as Telea's algorithm.

In \cite{FBTM07}, we left open the question for the well-posedness of our inpainting model. This question is positively answered in \cite{MyDiss}  as an application of
the theory presented in this paper.  
\medskip

\paragraph{Outline of the Paper}
Section \ref{Sect:Problem} is devoted to the full description of the linear problem \refEq{eqn:LinPDE} and its requirements.
The considered outflow sets $\Sigma$ are connected tree-like sets compactly contained in $\Omega$, thus
the biggest part of the requirements is concerned with the complicated structure of $\Sigma$ in order to define Lyapunov functions $T$ and transport fields $c$ which are
well-behaved close to and on the outflow set.
In section \ref{Sect:Coord} we discuss the non-linear coordinate system induced by the characteristics, in particular the features which are crucial for integral transformations later on.
In section \ref{Sect:LPExist}, we construct a solution $u \in \BV(\Omega)$ by using the method of characteristics and derive a priori bounds on the norm $\|u\|_{\BV(\Omega)}$ which depend only the data of the problem
and features of the Lyapunov function $T$. The main result is theorem \ref{Theo:ElemBV}.
Section \ref{Sect:Stab} is concerned with the uniqueness and the continuous dependence of $u$ on the data. Special attention is given to the non-linear dependence of the solution on the transport field $c$
and on the Lyapunov function $T$. The main result is theorem \ref{Theo:ContDepend}.
In section \ref{Sect:QuasiLin} we turn to the quasi-linear problem \refEq{eqn:QuasiLinPDE}. Exploiting the theory of the linear problem, in particular the continuous dependence on the data
and that we can close the gap as $u \in \BV(\Omega)$, we conclude the existence of a solution by the Schauder fixed point theorem.

\section{The Linear Problem and its Requirements}\label{Sect:Problem} % Requirements of the Linear Problem
The purpose of this section is to summarize all the requirements on the data of the linear problem:
\begin{equation}\label{eqn:LinProblem}
	\begin{aligned}
		 \SP< {c(x)}, {Du} > &= f(x) \cdot \Lm^2 \quad \mbox{ in } \quad \Omega \wo \Sigma  \;,  \\
		 u|_{\bd \Omega} &= u_0   \;, \\ 
		 \SP< {c(x)}, {\nabla T(x)} > & \geq \beta \cdot |\nabla T(x)|  \;.
	\end{aligned}
\end{equation}
The first set of requirements is concerned with the domain $\Omega$, the outflow set $\Sigma$ and the Lyapunov function $T$.
Later, in section \ref{Sect:LPExist}, we will construct the solution by the method of characteristics. For the characteristics we will see that
their time variable can be identified with the values of the Lyapunov function $T$ and that the outflow set $\Sigma$ is the location where
characteristics end. For these reasons, we call the Lyapunov function $T$ time function and the outflow set $\Sigma$ stop set.
The second set of requirements is concerned with the transport field $c$, the right-hand side $f$ and the boundary data $u_0$.
\medskip

\begin{req}\label{Req:Domain}(Domains)
	Domains $\Omega \subset \R^2$ are assumed to
	be open, bounded and simply connected
	and to have $\C^1$ boundary.
\end{req}
\medskip

Because of requirement \ref{Req:Domain} the boundary $\bd \Omega$ of a domain is a simple closed $\C^1$ curve. 
By $\gamma:\R \to \bd \Omega$  we denote  a generic regular and periodic parametrization of $\bd \Omega$.
That means $\gamma \in \C^1(\R,\bd \Omega)$ is surjective and $\gamma'(s) \neq 0 \; \forall s \in \R$. 
Furthermore, by $I=\IV[a,b[ \subset \R$ we denote an interval such that $\gamma|_I$ is a generator of $\gamma$.

Next, we collect the properties of admissible time functions.
For our problem time functions are global Lyapunov functions whose range corresponds to a finite time interval.
As time is usually a positive quantity which increases, we define the stop set to be the maximal level of $T$
while in literature the stop set is often the minimal level, see e.g. \cite{Amann:1990:ODE}.
\medskip

\begin{req}\label{Req:TimeFunc}
	(Time functions)
	Time functions are of type \(T: \Omega \to \R \). The upper level-sets of $T$ are denoted by
	\[
		\chi_{T \geq \lambda} := \{x \in \Omega : T(x) \geq \lambda\} \;.
	\]	
	We assume that time functions $T$ satisfy the following conditions:
	\smallskip
	\begin{enumerate}[1.]
		\item $T \in \C(\cl{\Omega})$.
			\smallskip
		\item The boundary of $\Omega$ is the start level: $T|_{\bd \Omega} = 0$. 
			\smallskip
		\item $T$ incorporates a stop set $\Sigma$ with stop time $1$:
			\begin{enumerate}[a)]
				\item $T(x)<1 \Leftrightarrow x \in \cl{\Omega} \wo \Sigma$.
				\item $T|_{\Sigma} = 1$, i.e., $\Sigma$ is the maximal level of $T$. 
			\end{enumerate}
			\smallskip
		\item $T$ increases strictly from $\bd \Omega$ towards $\Sigma$. That means that any upper level-set $\chi_{T \geq \lambda}$ is simply connected and
			\[
				\chi_{T \geq \lambda} = \cl{ \chi_{T > \lambda} } \qquad \forall \; \lambda \in \IV[0,1[ \;.
			\]
		\item Any proper upper level-set is a future domain: for every $\lambda \in [0,1[ $ the set $\chi_{T > \lambda}$ satisfies requirement \ref{Req:Domain}.	
			Furthermore, the field of interior unit normals to the $\lambda$-levels 
			\[
				\chi_{T=\lambda} = \bd \chi_{T > \lambda} \;, \qquad \lambda \in \IV[0,1[ 
			\]
			of $T$ is denoted by $N: \Omega \wo \Sigma \to S^1$. $N$ is required to be continuously differentiable and extendable onto $\bd \Omega$, i.e., $N \in \C^1(\cl{\Omega} \wo \Sigma)$. 
			\smallskip 
		\item[6.*] $T \in \C^2(\cl{\Omega})$, with $\nabla T (x)=0 \Leftrightarrow x \in \Sigma$.
	\end{enumerate}
\end{req}
\medskip

Remark on 6.*: this assumption is in order to ease things in the passages that follow. 
Because of 6.*, we obtain a simple description of the field $N$ on $\Omega \wo \Sigma$ by $	N(x) = \nabla T(x)/|\nabla T(x)|$ and
$N$ is continuously differentiable and extendable onto $\bd \Omega$.
\medskip

In part 3 of requirement \ref{Req:TimeFunc} we have assumed that $T$ features a stop set $\Sigma$. Here we state the geometric properties of allowed stop sets. 
\medskip

\begin{req}\label{Req:StopSet}
	(Stop sets)
	We assume that the stop sets $\Sigma$ satisfy the following conditions:
	\begin{enumerate}[1.]
		\item $\Sigma$ is a closed subset of  $\Omega$.
		\item $\Sigma$ is either an isolated point, or a connected set with tree-like structure (no loops).
		\item If $\Sigma$ is not an isolated point, it consists of finitely many rectifiable $\C^1$ arcs $\Sigma_k$:
			\[
				\Sigma = \bigcup\limits_{k=1}^{n} \Sigma_k \;.
			\]
			The collection $\{ \Sigma_k\}_{ k=1,\ldots,n}$ is assumed to be minimal in the number $n$ of arcs, so $\Sigma$ is decomposed by
			breaking it up at corners and branching points.
			\smallskip
			
			Furthermore, we require for each arc $\Sigma_k$ that its relative interior $\inner{\Sigma}_k$ has a given orientation by a 
			continuous unit normal $n_k : \inner{\Sigma}_k \to S^1$.
	\end{enumerate}
\end{req}
\medskip

In the case in which $\Sigma$ is not just an isolated point, we also need good behavior of the maps $T$ and $N$ close to and on the stop set $\Sigma$. 
For this purpose, we use the following concept of one-sided limits towards $z \in \inner{\Sigma}_k$:

If $x \in B_r(z) \wo \Sigma$ and if $r>0$ is sufficiently small, the projection $p$ of the point $x$ is unique. 
In view of this feature we say a point $x \in B_r(z) \wo \Sigma$
is on the right-hand side or plus side (respectively, on the left-hand or minus side) of $\inner{\Sigma}_k$ if
\begin{equation}
	\frac{x-p}{|x-p|} = +n_k(p)  \qquad \left( \frac{x-p}{|x-p|} = -n_k(p) \right) \;.
\end{equation}
Therewith, a sequence $(x_n)_{n \in \N}$, $x_n \in \Omega \wo \Sigma$, tends to $ z \in \inner{\Sigma}_k $ coming from the plus side 
(respectively, the minus side), in symbols
\begin{equation}
	x_n \to z_+ \qquad \left( x_n \to z_- \right) \;,
\end{equation}
if the sequence converges towards $z$ and almost all elements $x_n$ are on the plus side (respectively, minus side).

With the concept of one-sided limits, the good behavior of the maps $T$ and $N$ is summarized in the following requirement.
\medskip

\begin{req}\label{Req:TimeFuncSigma}	
	(Good behavior at $\Sigma$)
	\begin{enumerate}[1.]
		\item Requirements on $T$:
						
			Let $y \in \Sigma$ and $h \in S^1$ . Let $p=p(y,h)$ be the best possible order for the asymptotic formula
			\[
				T(y + r h) = 1 - \O(r^p) \; ,\quad r \to 0_+ \;.
			\]
			We require that there is a bound $q$ such that $\; \sup\limits_{y \in \Sigma} \sup\limits_{|h|=1} p(y,h) < q$. 
		\item Requirements on $N$:
			\begin{enumerate}[a)]
				\item $N$ has one-sided extensions onto the relatively open components $\inner{\Sigma}_k$ and those extensions are given by $\pm n_k$:
					\begin{align*}
						N^+(y) &\; := \lim\limits_{x \to y_+} N(x) \;, & N^+(y) &= -n_k(y) \;, \\
						N^-(y) &\; :=  \lim\limits_{x \to y_-} N(x) \;, & N^-(y) &= n_k(y) 
					\end{align*}
					for every  $ y \in \inner{\Sigma}_k$. 
					\smallskip
				\item The derivative $DN$ has one-sided extensions onto the relatively open components $\inner{\Sigma}_k$, i.e.,
					\begin{align*}
						(DN)^+(y) &\; := \lim\limits_{x \to y_+} DN(x) \;, &
						(DN)^-(y) &\; :=  \lim\limits_{x \to y_-} DN(x)  
					\end{align*}
					exist for every $y \in \inner{\Sigma}_k$.
					\smallskip
				\item $|DN| \in \L^1(\Omega)$, i.e., poles of $|DN|$ at corner-, branching- and terminal nodes of $\Sigma$ are integrable.
					This feature is assumed to hold in the case in which $\Sigma$ is an isolated point as well.
			\end{enumerate}
	\end{enumerate}
\end{req}
\medskip

So far we have the requirements on domains, stop sets, and time functions. Now we turn to the assumptions on transport fields. 
Here, for a given time function $T$ with stop set $\Sigma$, the causality of the transport field w.r.t. $T$ and its good behavior close to $\Sigma$ are the main concern.
\medskip

\begin{req}\label{Req:TransportField}
	(Transport fields)
	Transport fields of type $c:\Omega \wo \Sigma \to \R^2$ are required to satisfy:
	\medskip
	
	\begin{enumerate}[1.]
		\item $c \in \C^1(\Omega \wo \Sigma)^2$ and $c$ features the following properties:
			\begin{enumerate}[a)]
				\item $c$ and $Dc$ are continuously extendable onto $\bd \Omega $.
				\item If $\Sigma$ is not just an isolated point, then $c$ and $Dc$ have one-sided limits on the relatively open $C^1$ arcs
					$\inner{\Sigma}_k$ of $\Sigma$:
					\begin{align*}
						c^+(y) = & \; \lim\limits_{x \to y_+} c(x) \quad \mbox{and} \quad c^-(y) = \lim\limits_{x \to y_-} c(x) \; , \\
						(Dc)^+(y) = & \; \lim\limits_{x \to y_+} Dc(x) \quad \mbox{and} \quad (Dc)^-(y) = \lim\limits_{x \to y_-} Dc(x) \; ,
					\end{align*}
					for every $y \in \inner{\Sigma}_k$.
			\end{enumerate}
			\medskip
		\item Unit speed and causality condition:
			\begin{enumerate}[a)]
				\item $|c| = 1$.
				\item There is a lower bound $\beta > 0$ such that
					\[ 
						\beta \leq \SP< {c(x)}, {N(x)} > \leq 1 \qquad \forall x \in \cl{\Omega} \wo \Sigma \:.
					\]
				\item Conditions a) and b) hold for the one-sided limits as well, i.e.,
					\[
						|c^+(y)| = |c^-(y)| = 1
					\]
					and
					\[ 
						\beta \leq \SP< {c^+(y)}, {N^+(y)} > \leq 1 \quad , \quad \beta \leq \SP< {c^-(y)}, {N^-(y)} > \leq 1 \; , 
					\] 
					whenever $y$ belongs to some  $\inner{\Sigma}_k$.
			\end{enumerate}
			\medskip
		\item Let $z_k$, $k \in \{1, \ldots, m\}$ denote the terminal-, branching- and kink nodes of $\Sigma$. 			
			For every $\varepsilon > 0$ such that each disk $B_{\varepsilon}(z_k)$
			is compactly contained in $\Omega$, we define the set
			\[
				V_{\varepsilon} := \Sigma \; \cup \;  \bigcup\limits_{k=1}^m \cl{ B_{\varepsilon}(z_k) } \;.
			\]
			\begin{enumerate}[a)]
				\item For every admissible $\varepsilon > 0$, there is a bound $M_{\varepsilon}$ such that
					\[
						|Dc(x)| \leq M_{\varepsilon} \quad , \quad \forall \; x \in \Omega \wo V_{\varepsilon}
					\]
				\item $|Dc| \in \L^1(\Omega)$, i.e.,  poles of $|Dc|$ at $z_k$, $k \in \{1, \ldots, n\}$ are integrable.		
			\end{enumerate}
			\medskip
	\end{enumerate}
\end{req}
\medskip

Finally we need a right-hand side and boundary data.
\medskip

\begin{req}\label{Req:RHSandData}
	(Right-hand sides and boundary data)
	We assume right-hand sides $f: \Omega \to \R$ and Dirichlet boundary data $u_0 : \bd \Omega \to \R$ to be $f \in C^1(\cl{\Omega})$ and $u_0 \in \BV(\bd \Omega)$.
\end{req}
\smallskip

Remark: with $u_0 \in \BV(\bd \Omega)$ we mean that for every regular periodic parametri\-zation $\gamma$ of $\bd \Omega$, the pull-back $\gamma^* u_0 = u_0 \circ \gamma$ is a periodic $\BV$ function.
Moreover, because $\bd \Omega$ is one-dimensional, $\gamma^* u_0$ is $\BV$ function of one variable. Therefor the boundary data is essentially bounded.

In the following sections we assume, if not explicitly stated otherwise, that domains $\Omega$, stop sets $\Sigma$, time functions $T$, transport fields $c$, right-hand sides $f$, and boundary data $u_0$
satisfy the requirements above.
\medskip

\section{A Customized Coordinate System}\label{Sect:Coord}
The requirements of the previous section in combination are such that the family of characteristics which is the solution of the initial value problem
\begin{align*}
	y' &= c(y)\;, & y(0) &= x \in \Omega \wo \Sigma \;,
\end{align*}
gives us a customized coordinate system for problem \refEq{eqn:LinProblem}.
The subject of this section is the features of this non-linear coordinate system.
The first lemma is about the boundedness of the arc-length of characteristic curves. This property is important because the $\|.\|_{\BV}$-norm bound on the solution of \refEq{eqn:LinProblem} depends on it.
\medskip

\begin{lem}\label{Lem:ArcLenBound}
	\begin{enumerate}[a)]
		\item Let $q$ be the bound from requirement \ref{Req:TimeFuncSigma} part 1 , let $\varphi (t) := -t^{\frac{1}{q}}$ and let 
			\begin{equation} \label{eqn:TrafoT} 
				T_0(x) := 1+\varphi (1-T(x))  \; .
			\end{equation}
			Then the gradient $\nabla T_0$ of the transformed time function blows up at $\Sigma$ and is bounded below
			\[ 
				|\nabla T_0(x)| \geq m_0 > 0 \qquad \forall \; x \in \Omega \wo \Sigma \;.
			\]
			\smallskip
		\item For every regular $\C^1$ curve $x: \IV[0,a[ \to \cl{\Omega} \wo \Sigma$  ($a = \infty$ admissible)  that
			satisfies the condition 
			\begin{equation} \label{eqn:Ang1}
				0 < \beta \leq \left< N(x(\tau)), \frac{x'(\tau)}{|x'(\tau)|} \right> \:, \qquad \forall \tau \in \IV[0,a[ \; ,
			\end{equation}
			the arc length of $x$ is uniformly bounded by
			\begin{equation} \label{eqn:ArcLenBound} 
				\arclength(x) \leq  \frac{1}{\beta \cdot m_0}\:.
			\end{equation}
	\end{enumerate}
\end{lem}
\begin{proof} ~\smallskip
	\begin{enumerate}[a)]
		\item The function $T_0$ is well-defined since $0 \leq T \leq 1$ and its derivative  is
			\[
				\nabla T_0(x) = \varphi' (1 -T(x)) \cdot (-\nabla T(x)) =: H(x) \cdot \frac{\nabla T(x)}{|\nabla T(x)|} \; ,
			\]
			with
			\[
				H(x) :=  \frac{1}{q} (1 -T(x))^{\frac{1-q}{q}} \cdot |\nabla T(x)| > 0 \;, \qquad x \in \Omega \wo \Sigma \;.
			\]
			Let $y \in \Sigma$, $h \in S^1$ and $r>0$. Then, by requirement \ref{Req:TimeFuncSigma} part 1 we have
			\begin{align*}
				1-T(y+rh) &= C_1 r^p\;, & p &=p(y,h) \; , \; C_1 > 0 \;.
			\end{align*}
			Because $T \in C^2(\Omega)$ and $\nabla T|_{\Sigma} = 0$ we have 
			\begin{align*}
				|\nabla T(y+rh)| &= C_2 r^{p-1} \;,
			\end{align*}
			with the same $p$ as before.
			Combining these asymptotic formulae, we obtain
			\[
				H(y+rh) = C_3  r^{\frac{p(1-q)}{q}} r^{p-1} = C_3 r^{\frac{p-q}{q}} \:.
			\]
			Since, by requirement \ref{Req:TimeFuncSigma} part 1, $q > p(y,h)$ holds uniformly, we get a blow-up
			\[
				\lim\limits_{r \to 0_+} |\nabla T_0(y+rh)| = \lim\limits_{r \to 0_+} H(y+rh) = \infty \:, 
			\]
			for any choice of $h \in S^1$ and every $y \in \Sigma$.
			
			We show next that $|\nabla T_0| \geq m_0 > 0$.  
			Assume by contradiction that
			\[
				\inf\limits_{x \in \cl{\Omega} \wo \Sigma } H(x) = \inf\limits_{x \in \cl{\Omega} \wo \Sigma } |\nabla T_0(x)| = 0 \: .
			\]
			and choose an open neighborhood $U$ of $\Sigma$ such that $H|_U \geq M$,
			for some constant $M > 0$, which is possible because of the blow-up. Then the restriction onto the compact complement 
			$\hat{H} = H|_{\cl{\Omega} \wo U}$, 
			being a continuous function,  must take the minimum
			\[
				\min\limits_{x \in \cl{\Omega} \wo U } \hat{H}(x) = 0 
			\]
			at some point $\hat{x} \in \cl{\Omega} \wo U$. But then, the definition of $H$ implies
			\[
				\hat{H}(\hat{x}) = H(\hat{x}) = 0 \quad \Rightarrow \quad |\nabla T(\hat{x})| = 0 \:,
			\]
			which is a contradiction, since $\hat{x} \notin \Sigma$. Thus, $H=|\nabla T_0|$ has a minimum $m_0$ which is greater than zero.
		\item Using $T_0$ we estimate the arc length from above by
			\begin{align*} 
				1 \geq T_0(x(t))-T_0(x(0)) &= \int\limits_{0}^{t} \SP< {\nabla T_0(x(\tau))} , {x'(\tau)} > \: d\tau \\
				&= \int\limits_{0}^{t} \SP< {N(x(\tau))} , {\frac{x'(\tau)}{|x'(\tau)|}}  > \cdot |\nabla T_0(x(\tau))|  \cdot |x'(\tau)|\: d\tau \\
				&\geq  \beta \int\limits_{0}^{t} |\nabla T_0(x(\tau))| \cdot |x'(\tau)|\: d\tau  \; \geq \; \beta \cdot m_0 \int\limits_{0}^{t}  |x'(\tau)|\: d\tau \: .
			\end{align*} 
			The limit $t \to a$ finally yields the uniform bound on the arc length of such curves 
			\[
				\arclength(x) \leq \frac{1}{\beta \cdot m_0}  \; ,
			\]
			which depends only on $\beta$ and information from $T$. 
	\end{enumerate}
	\hfill
\end{proof}
\smallskip

Because of its nice properties, the transformed version $T_0$ defined by equation \refEq{eqn:TrafoT} will be identified -- instead of $T$ --
with the time variable of the characteristics. This is the subject of the next lemma.
Whenever we speak about $T_0$ we mean this transformed version of a given time function $T$.
\medskip

\begin{lem} \label{Lem:RadialTrafo}
	\begin{enumerate}[a)]
		\item The initial value problem
			\begin{align}
				y' &= c(y) \;, & y(0) &= x \in \Omega \wo \Sigma \; , \label{eqn:IVPfor}
			\end{align}
			has a unique maximally continued solution  $y: \IV]{t_-},{t_+}[ \to \R^2 $, with $ -\infty  < t_- < 0 < t_+ < \infty$.
			\smallskip
	
			Every trajectory $y$ connects the sets $\bd \Omega$ and $\Sigma$, i.e.,
			\begin{align*}
				\lim\limits_{t \to t_-} y(t) &\in \bd \Omega \;,&  \lim\limits_{t \to t_+} y(t) &\in \Sigma \; .
			\end{align*}
			For every point $z \in \inner{\Sigma}_k$ in the relative interior of some $C^1$ arc of $\Sigma$, there are exactly two
			trajectories which hit $z$ in the limit $t \to t_+$, one for each side of $\inner{\Sigma}_k$.
			\medskip
		\item The solution $y_0$ of the velocity-transformed IVP
			\begin{align}
				y' &= c_0(y) \;, & y(0) &= x \in \Omega \wo \Sigma \; , & c_0 &:= \frac{c}{\SP<{c}, {\nabla T_0} >} \;, \label{eqn:IVPtrans}
			\end{align}
			satisfies $\; T_0( y_0 (t) ) = t + T_0(x)$.
	\end{enumerate}
\end{lem}
\medskip

\begin{proof} ~\smallskip
	\begin{enumerate}[a)]
		\item We consider the IVP \refEq{eqn:IVPfor}.
			Because $c$ is Lipschitz continuous by requirement \ref{Req:TransportField} part 3a) there exists a maximally continued, 
			unique solution $y$ with time domain $\IV]{t_-},{t_+}[$ and $0 \in \IV]{t_-},{t_+}[$.
			\smallskip
			 
			Because of the unit speed condition $|c| = 1$, $y$ never stops inside $ \Omega \wo \Sigma $ and never blows up.
			The causality condition (requirement \ref{Req:TransportField} part 2b) ) implies, by
			\begin{equation} \label{eqn:Increase}
				 \frac{d}{dt} T_0( y(t) ) = \SP< {\nabla T_0( y(t) )}, {c(y(t))} >  \geq  m_0 \cdot \beta > 0 \;,
			\end{equation}
			that $T_0( y(t) )$ strictly increases at least at a rate of $m_0 \cdot \beta$.\\
			Thus, $y$ collapses at the boundary of $ \Omega \wo \Sigma $ and by (\ref{eqn:Increase}) there follows:
			\begin{itemize}
				 \item Going forward $t \to t_+$: collapse at $\Sigma$ after finite time $t_+$ ,
				 \item Going backward $t \to t_-$: collapse at $\bd \Omega$ after finite time $t_-$.
			 \end{itemize}
			 Because of unit speed the values $t_+$ , $t_-$ are exactly the arc lengths, which are finite by lemma \ref{Lem:ArcLenBound}.
			 \smallskip
			 
			Assume now that  $z \in \inner{\Sigma}_k$ and consider the ''plus side'', i.e., the side  where $n_k(z)$ points.
			Since $c$ and $Dc$ both extend from the plus side  onto $\inner{\Sigma}_k$ by $c^+$ and $(Dc)^+$,  the backward IVP
			\begin{align*}
				y' &= -c(y) \;, & y(0) &= z  \; , \qquad \mbox{with} \; c(z) := c^+(z) \;.
			\end{align*}
			has a unique solution that starts at $z \in \inner{\Sigma}_k$ and evolves away from $\Sigma$ into the plus side. 
			Hence, vice versa there is only one solution $y$ of the forward IVP \refEq{eqn:IVPfor} that comes from the plus side,
			heads for $z \in \inner{\Sigma}_k$, and hits $z$ in the end. The same argument holds for the minus side.
		\item We consider again the forward IVP \refEq{eqn:IVPfor},  the initial value $x$ of which satisfies $T_0(x) < 1 = \max\limits_{z \in \Omega} T_0(z)$. 
			For $\lambda $ with $T_0(x) \leq \lambda < 1$, there is a unique time $\tau(\lambda)$, when the solution $y$ of \refEq{eqn:IVPfor} 
			crosses the $\lambda$-level of $T_0$. 
			This is because $T_0( y(t) )$ increases strictly by \refEq{eqn:Increase}, so $y$ crosses the $\lambda$-level only once .
			
			Then, viewing $\tau$ as a function of $\lambda$, the implicit function theorem applied to
			\[ 
				T_0 ( y( \tau ) ) = \lambda
			\]
			yields the differentiability of $\tau$ w.r.t. $\lambda$ and the derivative:
			\[ 
				\tau'(\lambda) = \left. \frac{1}{\SP< {\nabla T_0(z)} , {c(z)} >} \right|_{z=y( \tau(\lambda) )} \; .
			\]
			Using $\tau$, we now change the independent variable
			\[
				y_0(\lambda ) := y(\tau(\lambda + \lambda_0) ) \;,
			\]
			where $\lambda_0 := T_0(x)$ and $\tau(\lambda_0) = 0$.
			Then, $y_0$ satisfies the initial condition $y_0(0) = x$ and has the derivative
			\begin{align*}
				y_0' (\lambda ) &= y' (\tau(\lambda + \lambda_0) ) \cdot \tau'(\lambda + \lambda_0) 
				 = \left. \left( c(z)  \cdot  \frac{1}{\SP< {\nabla T_0(z)} , {c(z)} >} \right) \right|_{z=y( \tau(\lambda + \lambda_0) )} \\
				& = \left.  \frac{c(z)}{\SP< {\nabla T_0(z)} , {c(z)} >} \right|_{z=y_0( \lambda)} = c_0(y_0(\lambda)) \;.
			\end{align*}
			Consequently, $y_0$ is the unique solution of the transformed IVP \refEq{eqn:IVPtrans}
			and --- by construction --- satisfies 
			\[ 
				T_0(y_0(\lambda)) = T_0(y(\tau(\lambda + \lambda_0) )) = \lambda + \lambda_0 = \lambda + T_0(x).
			\]			
	\end{enumerate}
	\hfill
\end{proof}
\smallskip

By lemma \ref{Lem:RadialTrafo} part b), we get the useful property that -- when using $c_0$ instead of the original transport field $c$ --
the time variable of a characteristic $y_0$ is given by $T_0$. We use this feature to introduce new coordinates on $\Omega \wo \Sigma$ whose concept is similar to
polar coordinates on a disk: the level lines of $T_0$ will play the role of the concentric circles and the characteristic curves that of the radial lines.

Whenever we speak about $c_0$ we mean the transformed version of a given transport field $c$ according to lemma \ref{Lem:RadialTrafo} part b).
\medskip

\begin{cor} \label{Cor:PolarCoord}
	(Polar coordinates)
	Let $\gamma$ be a regular periodic parametrization of $\bd \Omega$ with generator $\gamma|_{I}$, $I=[a,b[$.
	\begin{enumerate}[a)]
		\item	The general solution $\xi(t,s)$ of the forward IVP
			\begin{align*}
				y' &= c_0(y) \; , & y(0) &= \gamma(s) \;,
			\end{align*}
			defines a diffeomorphism $\; \xi : \IV]0,1[ \times \IV]a,b[ \to \Omega \wo (\Sigma \cup S) \; $ , where
			\[ 
				S =  \left\{ \xi(t,a) : t \in \IV]0,1[ \right\} \; .
			\]	
		\item	Let $\eta(t,x)$ denote the general solution of the backward IVP
			\begin{align*}
				y' &= -c_0(y) \; , & y(0) &= x \in \Omega \wo (\Sigma \cup S)\; .
			\end{align*}
			Then the inverse map $\xi^{-1}(x) = (t(x),s(x))^T$ is given by
			\begin{align*}
				\xi^{-1}(x) &=  ( \; T_0(x) \;,\; \gamma^{-1}(\eta(T_0(x) , x)) \; )^T\; , &  x& \in \Omega \wo (\Sigma \cup S) \;.
			\end{align*}	
			The relation between $\xi$ and $\eta$ is
			\begin{equation}\label{eqn:RelFBChar}
				\xi(t,s(x)) = \eta(\; T_0(x) - t \; , \; x \;)   \; .
			\end{equation}
		\item	If $\gamma$ is oriented clockwise, the Jacobian $D\xi = (\pd_t \xi | \pd_s \xi)$ of the diffeomorphism $\xi$ has a positive determinant 
			and  the estimate
			\begin{equation}\label{eqn:RelDet}
				0 < \det D\xi \leq |\partial_t \xi| |\partial_s \xi| \leq \frac{\det D\xi}{\beta} 
			\end{equation}
			holds true.		
			If $\gamma$ is counter-clockwise, the assertions hold for $-\det D\xi$.
			\medskip
		\item For each of the relatively open $\C^1$ arcs $\inner{\Sigma}_k$ of $\Sigma$, we can find -- w.r.t. the orientation $n_k$ -- two subsets $J_{k,+}$ and $J_{k,-}$ of $I$ 
			such that the maps $\xi(1,s)$ with $s \in J_{k,+}$ and $\xi(1,s)$ with $s \in J_{k,-}$ are both regular $\C^1$-parametrizations of $\inner{\Sigma}_k$.
	\end{enumerate}
\end{cor}
\smallskip

\begin{proof}
	The proof of these statements can be found in \cite[section 3.2]{MyDiss}. \hfill
\end{proof}
\medskip

Our customized coordinate system is now given by the diffeomorphism $\xi$ of corollary \ref{Cor:PolarCoord} part a) and we can turn to solving problem \refEq{eqn:LinProblem}.
\medskip
			
\section{Existence of a Solution}\label{Sect:LPExist}
In this section we construct a candidate solution to the linear problem \refEq{eqn:LinProblem} and study its properties.
The tool for construction is the method of characteristics (see e.g. \cite{Evans:PDE} or \cite{Zauderer:PDE}) which we apply it to the scaled PDE  
\begin{equation} \label{eqn:TPDEtrans}
	\SP< {c_0(x)}, {Du} > = f_0(x) \cdot \Lm^2 \quad \mbox{ in } \quad \Omega \wo \Sigma \;,
\end{equation}
with
\begin{align}\label{eqn:DeffNull}
	 c_0(x) &:= \frac{c(x)}{\SP< {c(x)}, {\nabla T_0(x)} >} \qquad \mbox{and} & f_0(x) &:= \frac{f(x)}{\SP< {c(x)}, {\nabla T_0(x)} >} \;.
\end{align}
For PDE \refEq{eqn:TPDEtrans} the characteristic equation is exactly the forward IVP from corollary \ref{Cor:PolarCoord} a)
\begin{align*} 
	y' &= c_0(y) \;, & y(0) &= \gamma(s) \:, 
\end{align*}
and the family of forward characteristics is $\xi$. Clear\-ly, the solution $\eta$ of the corresponding backward IVP
is the family of backward characteristics.

Let now $v(t,s) := u \circ \xi(t,s) $. Then  -- at least formally -- the partial derivative of $v$ w.r.t. $t$ is given by
\[ 
	\partial_t v(t,s) = \left< \nabla u \circ \xi(t,s) , \partial_t \xi(t,s) \right> = \left< \nabla u  , c_0 \right> \circ \xi(t,s) = f_0 \circ \xi(t,s) \:, 
\]
together with the initial condition
\[
	v(0,s) = u(\xi(0,s)) = u(\gamma(s)) = \gamma^*u_0(s) \;.
\]
Here, $\gamma^*$ denotes the pull-back operation.
Thus, by the fundamental theorem of calculus, we obtain
\begin{equation}\label{eqn:SolChar}
	v(t,s) = \gamma^*u_0(s) + \int\limits_0^t f_0 \circ \xi(\tau,s) \: d\tau \quad =: v_1(t,s) + v_2(t,s)\; .
\end{equation}
The function $v$ represents our candidate solution $u$ in characteristic variables $(t,s)$. By using the inverse map $\xi^{-1}$ from corollary \ref{Cor:PolarCoord} b)
and the relation \refEq{eqn:RelFBChar} between $\xi$ and $\eta$, we push $v$ forward
onto  $\Omega \wo \Sigma$ to have $u = v \circ \xi^{-1}$ in the original variables $x$,
\begin{equation}\label{eqn:SolOrig} 
	u(x) = u_0( \eta(T_0(x),x)) + \int\limits_0^{T_0(x)} f_0 \circ \eta(\tau ,x) \: d\tau \quad =: u_1(x) + u_2(x)\;. 
\end{equation}
\smallskip

Equations \refEq{eqn:SolChar} and \refEq{eqn:SolOrig} present our candidate solution.
The next theorem shows that the candidate belongs to the space $\BV(\Omega) \cap \L^{\infty}(\Omega)$.
\medskip

\begin{theo} \label{Theo:ElemBV}
	(Element of $\BV$)
	The candidate solution $u$ from \refEq{eqn:SolOrig}  with its decomposition $u = u_1 + u_2$ has the properties:
	\medskip
	\begin{enumerate}[a)]
		\item $u$ is an element of  $\BV(\Omega \wo \Sigma) \cap \L^{\infty}(\Omega) $.	
			Its $\L^{\infty}(\Omega)$ norm is bounded by
			\[
				\|u\|_{\L^{\infty}(\Omega)} \leq \|u_0\|_{\L^{\infty}(\bd \Omega)} + \frac{\|f\|_{\infty}}{\beta \cdot m_0}  \;. 
			\]	
			The derivative measure of $u$ is 
			\[
				Du = c_0^{\perp}(x) \cdot \mu \; + \; \nabla u_2(x) \cdot \Lm^2 \quad \mbox{with} \quad \mu := \xi_{\sharp}\left( \Lm^1 \otimes D\gamma^*u_0\right) \;,
			\]
			whereas the total variation is bounded by
			\begin{align*}
				|Du|(\Omega \wo \Sigma) \leq M_{\Omega \wo \Sigma} := \quad & \frac{|Du_0|(\bd \Omega)}{\beta \cdot m_0} + \left( \frac{\|f\|_{\infty}}{\beta} + 
				\frac{\|\nabla f \|_{\infty}}{\beta^2 \cdot m_0}\right) \cdot \Lm^2(\Omega) \\
				&+ \frac{\|f\|_{\infty}}{\beta^3 \cdot m_0} \cdot \left( \|Dc\|_{L^1(\Omega)} + \|DN\|_{L^1(\Omega)}\right) \;.
			\end{align*}
		\item $u$ extends onto $\Sigma$, i.e., $u$ is an element of  $\BV(\Omega) \cap L^{\infty}(\Omega)$.			
			In comparison to part a) the extension introduces -- in the derivative $Du$ -- an additional jump part for every $C^1$ arc $\Sigma_k$ of $\Sigma$ : 
			\begin{align*}
				Du = \; & \; c_0^{\perp}(x) \cdot \mu \; + \; \nabla u_2(x) \cdot \Lm^2 
				 + \sum\limits_{k=1}^n (u_{\Sigma_k}^+(x)-u_{\Sigma_k}^-(x)) \; n_k(x) \cdot \Hm^1 \rto \Sigma_k \;,
			\end{align*}
			where $u_{\Sigma_k}^-$ and $u_{\Sigma_k}^+$ are the left and right interior $\BV$-traces of $u$ on $\Sigma_k$.			
			The bound on the total variation is added up by
			\[
				|Du|(\Omega) \leq M_{\Omega \wo \Sigma} + 2 \cdot \|u\|_{\L^{\infty}(\Omega)} \cdot \Hm^1(\Sigma) \;.
			\]  
	\end{enumerate}
\end{theo}
\medskip

\begin{proof}  
	Before going into the details of the proof,  we summarize some facts concerning the change of variables.
	If $\varphi \in \C_c^1(\Omega)$ or $\varphi \in \C_c^1(\Omega \wo \Sigma)$ is a test function, we will denote by $\psi(t,s) := \varphi \circ \xi(t,s)$ 
	the test function in characteristic variables. By the chain rule we then obtain the derivative with respect to $(t,s)$ 
	\[ 
		\nabla_{t,s} \psi = D\xi^T \cdot \nabla_x \varphi \circ \xi \quad \Rightarrow
		\quad \det D\xi \cdot \nabla_x \varphi \circ \xi = \left(-\partial_s \xi^{\perp} | \partial_t \xi^{\perp} \right) \cdot \nabla_{t,s} \psi\; .
	\]
	Let $l=l(k)$ be the non-trivial permutation of $\{1,2\}$. Then we write for the $k$-th component
	\begin{equation}\label{eqn:ChangeVarDeriv}
		\begin{aligned}
			\det D\xi \cdot \pd_k \varphi \circ \xi 	&= (-1)^l \cdot ( \; \pd_t( \pd_s \xi_l \cdot \psi) - \pd_s( \pd_t \xi_l \cdot \psi ) \; ) \; .  
		\end{aligned}
	\end{equation}
	This equality can easily be derived from the product rule.
	
	Without loss of generality, we assume that the parametrization of the boundary $\gamma(s) = \xi(0,s)$ is clockwise which implies $\det D\xi > 0$.
	In addition, we have the useful relations
	\begin{equation}\label{eqn:XiGeo}
		\pd_t \xi^{\perp} = c_0^{\perp} \circ \xi \;, \qquad \pd_ s \xi^{\perp} = - N\circ \xi \cdot |\partial_s \xi| \;.
	\end{equation}
	The first equality is clear by the characteristic ODE. The second is a consequence of $\xi(t,.)$ being a clockwise parametrization of the $t$-level of $T_0$.
	Finally, we remark that $\psi(t,s)$ is periodic w.r.t. the variable $s$, since $\xi$ is. Thus we have $\psi(t,a) = \psi(t,b)$.
	\medskip
	
	Now we can compute the derivative measure $Du$, which in both parts is the same process. Let $\varphi \in \C_c^1(\Omega)$ be a test function.
	For  the moment we restrict the discussion to subsets $\Omega_{\lambda}$ of $\Omega$  which are lower level-sets of $T_0$, that is
	\[ 
		\Omega_{\lambda} := \Omega \cap \{ x \in \Omega : T_0(x) \leq \lambda\} \;,
	\]
	for $0 < \lambda < 1$. Note that $\Omega_{1} = \Omega $.	
	When later on we have $\varphi \in \C_c^1(\Omega \wo \Sigma)$ (respectively $\varphi \in \C_c^1(\Omega \wo \Sigma)^2$),
	as is the case for part a), we will choose $\lambda$ big enough such that $\supp \varphi \subset \Omega_{\lambda}$.
	For part b) we will pass to the limit $\lambda \to 1$ instead. 
	\medskip
	
	We separately compute the derivatives of $u_1$ and $u_2$. 
	In order to get $D_k u_1$ we have to look at the following integral: 
	\begin{align*}
		\int\limits_{\Omega_{\lambda}} u_1(x) \; \pd_k \varphi(x) \; dx 
		&= \int\limits_a^b \int\limits_0^{\lambda} v_1(t,s) \; \pd_k \varphi \circ \xi(t,s) \; \det D\xi(t,s) \; dt \;ds \\
		&= (-1)^l \int\limits_a^b \int\limits_0^{\lambda} \gamma^*u_0(s) \; ( \; \pd_t( \pd_s \xi_l \cdot \psi) 
			- \pd_s( \pd_t \xi_l \cdot \psi ) \; ) \; dt \;ds \,.
	\end{align*}
	By changing the order of integration and using the integration by parts formula for functions of one variable, one obtains
	\begin{align*}
		=& \, (-1)^l \int\limits_a^b \gamma^*u_0(s)  \int\limits_0^{\lambda}   \pd_t( \pd_s \xi_l \cdot \psi) \, dt \,ds 
			- (-1)^l \int\limits_0^{\lambda}  \int\limits_a^b \gamma^*u_0(s) \; \pd_s( \pd_t \xi_l \cdot \psi )  \, ds \,dt \\
		=& \, (-1)^l \int\limits_a^b  \gamma^*u_0(s) \;  \left[ \pd_s \xi_l \cdot \psi \right]_{t=0}^{\lambda}  \,ds 
		   \;	+ \; (-1)^l  \int\limits_0^{\lambda}  \int\limits_a^b \pd_t \xi_l \cdot \psi   \, dD\gamma^*u_0(s)   \,dt \;.
	\end{align*}
	In the last step we used the fact that $\gamma^*u_0$ is a periodic $\BV$-function. Because $\varphi$ has compact support in $\Omega$, we have furthermore
	$\psi(0,s) = \varphi \circ \xi(0,s) = \varphi(\gamma(s)) = 0$,
	so the result reduces to
	\begin{align*}
		=& \; (-1)^l \int\limits_a^b v_1(\lambda,s) \;  \pd_s \xi_l(\lambda,s) \cdot \psi(\lambda,s)  \,ds +
			 (-1)^l \int\limits_a^b \int\limits_0^{\lambda} \pd_t \xi_l \cdot \psi  \,dt \, dD\gamma^*u_0(s)  \,.
	\end{align*}
	For the vector-valued version we test with $\varphi \in \C_c^1(\Omega)^2$ and obtain
	\begin{align*}
		& \int\limits_{\Omega_{\lambda}} u_1(x) \; \div \varphi(x) \; dx = \int\limits_{\Omega_{\lambda}} u_1(x) \; \pd_1 \varphi_1(x) \; dx + \int\limits_{\Omega_{\lambda}} u_1(x) \; \pd_2 \varphi_2(x) \; dx \\
		&\, = \, - \, \int\limits_a^b\;  - \pd_s \xi_2(\lambda,s) \cdot \psi_1(\lambda,s) \; v_1(\lambda,s)  \,ds  \, - \, \int\limits_a^b \int\limits_0^{\lambda}\,  -\pd_t \xi_2 \cdot \psi_1  \,dt \, dD\gamma^*u_0(s)  \\
		&\qquad   - \, \int\limits_a^b\;   \pd_s \xi_1(\lambda,s) \cdot \psi_2(\lambda,s) \; v_1(\lambda,s)  \,ds  \, - \, \int\limits_a^b \int\limits_0^{\lambda}\,  \pd_t \xi_1 \cdot \psi_2  \,dt \, dD\gamma^*u_0(s)  \\
		&\, =  \, - \, \int\limits_a^b  \;  \SP< {\psi(\lambda,s)} , {\pd_s \xi^{\perp}(\lambda,s)} > \; v_1(\lambda,s)  \,ds
		  \, - \, \int\limits_a^b \int\limits_0^{\lambda}\;  \SP<  {\psi }, {\pd_t \xi^{\perp} }>  \,dt \, dD\gamma^*u_0(s)  \,.
	\end{align*}
	By using  the relations \refEq{eqn:XiGeo}, the last result can be written as
	\begin{align*}
		\int\limits_{\Omega_{\lambda}} u_1(x) \; \div \varphi(x) \, dx 
		=& \, - \, \int\limits_a^b  \;  \SP< {\varphi}  , {- N} > \circ \xi(\lambda,s)\; u_1 \circ \xi(\lambda,s) \; |\partial_s \xi(\lambda,s)| \,ds \\
		& \, - \, \int\limits_a^b \int\limits_0^{\lambda}\;  \SP<  {\varphi} , {c_0^{\perp}} >   \circ \xi \,dt \, dD\gamma^*u_0(s)  \,.
	\end{align*}
	Here, the first summand integrates w.r.t. the $\Hm^1$-measure along the $\lambda$-level of $T_0$. For the restriction of measures onto $\lambda$-levels of $T_0$ we will use the abbreviation 
	\[ 
		\Hm^1 \rto \lambda := \Hm^1 \rto \{x \in \Omega : T_0(x) = \lambda\} \;.
	\]
	In the second integral we rechange variables by pushing forward the product measure $\Lm^1 \otimes D\gamma^*u_0$ with the diffeomorphism $\xi$.
	Let $\mu$ denote the pushed-forward measure $\mu := \xi_{\sharp} (\Lm^1 \otimes D\gamma^*u_0)$.
	Then we finally obtain
	\begin{align*}
		\int\limits_{\Omega_{\lambda}} u_1(x) \; \div \varphi(x) \; dx 
		=&  -  \int\limits_{\Omega}   \SP< {\varphi(x)} , {- N(x)}  > \; u_1(x) \;d\Hm^1 \rto \lambda(x) \\
		& -  \int\limits_{\Omega_{\lambda}} \SP< {\varphi (x)} , {c_0^{\perp} (x)}>  \;d\mu(x)  \,.
	\end{align*}
	For the derivative of $u_2$ we perform the same steps as above with the integral
	\begin{align*}
		\int\limits_{\Omega_{\lambda}} u_2(x) \; \pd_k \varphi(x) \; dx 
		&= (-1)^l \int\limits_a^b \int\limits_0^{\lambda} v_2(t,s) \; ( \; \pd_t( \pd_s \xi_l \cdot \psi) - \pd_s( \pd_t \xi_l \cdot \psi ) \; ) \; dt \;ds \;.
	\end{align*}
	After changing the order of integration we go on with integration by parts:
	\begin{align*}
		=& \; (-1)^l \int\limits_a^b \int\limits_0^{\lambda} v_2(t,s) \;  \pd_t( \pd_s \xi_l \cdot \psi) \; dt \;ds  - (-1)^l \int\limits_0^{\lambda}  \int\limits_a^b v_2(t,s) \; \pd_s( \pd_t \xi_l \cdot \psi )  \; ds \;dt \\
		=& \; (-1)^l \int\limits_a^b  \left( v_2(\lambda,s) \;  \pd_s \xi_l(\lambda,s) \cdot \psi(\lambda,s)  - \int\limits_0^{\lambda} \pd_t v_2 \;  \pd_s \xi_l \cdot \psi \; dt \right) \;ds  \\
		  & \quad	- \; (-1)^l  \int\limits_0^{\lambda} \int\limits_a^b -\pd_s v_2 \;  \pd_t \xi_l \cdot \psi   \; ds  \;dt \\
		=& \; (-1)^l \int\limits_a^b v_2(\lambda,s) \;  \pd_s \xi_l(\lambda,s) \cdot \psi(\lambda,s)  \;ds 
		- \int\limits_a^b \int\limits_0^{\lambda} (-1)^l( \pd_s \xi_l \; \pd_t v_2 - \pd_t \xi_l \; \pd_s v_2 \;) \cdot \psi  \,dt \, ds \,.
	\end{align*}
	In the second equality we have used the 	periodicity of $\xi$ w.r.t. $s$.
	Because of $v_2 = u_2 \circ \xi$ we can -- according to equation \refEq{eqn:ChangeVarDeriv} -- substitute the last integrand to get
	\begin{align*}
		=& \; (-1)^l \int\limits_a^b v_2(\lambda,s) \;  \pd_s \xi_l(\lambda,s) \cdot \psi(\lambda,s)  \;ds 
		   \; - \; \int\limits_a^b \int\limits_0^{\lambda} \det D\xi \cdot  \pd_k u_2 \circ \xi \cdot \psi  \;dt \; ds \;. 
	\end{align*}
	By means of the last result and a rechange of variables, we end up with 
	\begin{align*}
		\int\limits_{\Omega_{\lambda}} u_2(x) \; \div \varphi(x) \; dx 
		=& \; - \; \int\limits_{\Omega}  \;  \SP< {\varphi(x)} , {- N(x)} > \; u_2(x) \;d\Hm^1 \rto \lambda(x) \\
		& \; - \; \int\limits_{\Omega_{\lambda}}\;  \SP<  {\varphi (x)} , {\nabla u_2 (x) } >  \;dx  \;,
	\end{align*}
	when testing with $\varphi \in \C_c^1(\Omega)^2$.	
	Adding the partial results for $u_1$ and $u_2$ gives us
	\begin{equation} \label{eqn:DuMeas1} 
		\begin{aligned}
			\int\limits_{\Omega_{\lambda}} u(x) &  \; \div \varphi(x) \; dx 
			= \; - \; \int\limits_{\Omega}  \;  \SP<  {\varphi (x)} , {-N(x) } > \; u(x) \;d\Hm^1 \rto \lambda(x)  \\
			&\quad  \; - \; \int\limits_{\Omega_{\lambda}}\;  \SP<  {\varphi (x)} , {c_0^{\perp} (x) } >  \;d\mu(x) \; - \; \int\limits_{\Omega_{\lambda}}\;  \SP<  {\varphi (x)} , {\nabla u_2 (x) } >  \;dx  \;. 
		\end{aligned}
	\end{equation}	
	Now, we are ready to turn to the proof of the assertions a) and b).
	\begin{enumerate}[a)]
		\item In this part we have $\Omega \wo \Sigma$ as the domain of $u$. If we test with $\varphi \in \C_c^1(\Omega \wo \Sigma)^2$, we can choose $\lambda < 1$  large enough that
			equation \refEq{eqn:DuMeas1} reduces to
			\begin{align*}
				\int\limits_{\Omega} u(x)   \; \div \varphi(x) \; dx 
				&=  \; - \; \int\limits_{\Omega}\;  \SP<  {\varphi (x) \;} , {\; c_0^{\perp} (x) \;d\mu(x) + \nabla u_2 (x)   \;dx} >  \;. 
			\end{align*}
			Consequently, the derivative measure is given by $Du = c_0^{\perp} (x) \cdot \mu + \nabla u_2 (x) \cdot \Lm^2$.
			What remains to show is the boundedness of $\|u\|_{\BV(\Omega \wo \Sigma)}$.
			\smallskip
			
			Because $u_0 \in \BV(\bd \Omega)$ is a $\BV$ function of one variable, $u_0$ is essentially bounded (see \cite[chapter 3.2]{AmbrosioBV}).
			By equation \refEq{eqn:SolOrig} then, we obtain a bound on the $\L^{\infty}$-norm of $u$: 
			\[
				\|u\|_{\L^\infty(\Omega)} \leq \| u_0 \|_{\L^\infty(\bd \Omega)} + \|T_0\|_{\infty} \| f_0 \|_{\infty}  \leq \| u_0 \|_{\L^\infty(\bd \Omega)} +  \frac{\| f \|_{\infty} }{\beta \cdot m_0} \;.
			\]
			For the total variation $|Du|(\Omega \wo \Sigma)$ we estimate both summands of $Du$ separately, beginning with
			\begin{align*}
				&\left| \int\limits_{\Omega}\;  \SP<  {\varphi (x)} , {c_0^{\perp} (x)} >  \;d\mu(x) \right| 
				 \; = \; \left| \int\limits_a^b \int\limits_{0}^{1}  \SP<  {\varphi \circ \xi} , {c_0^{\perp} \circ \xi}>   \; dt \; dD\gamma^*u_0(s) \right| \\
				& %\leq \;  \int\limits_a^b \int\limits_{0}^{1}  |c_0 \circ \xi|  \; dt \; d|D\gamma^*u_0|(s) \cdot \| \varphi \|_{\infty}
				 \leq \; \frac{1}{\beta \cdot m_0} \cdot \int\limits_a^b \; d|D\gamma^*u_0|(s) \cdot \| \varphi \|_{\infty} \; = \; \frac{|Du_0|}{\beta \cdot m_0} \cdot   \| \varphi \|_{\infty} \;.
			\end{align*}
			Clearly, the total variation of this summand is bounded by
			$\left|c_0^{\perp}(x) \cdot \mu \right|(\Omega \wo \Sigma) \leq \frac{|Du_0|(\bd \Omega)}{\beta \cdot m_0}$,
			which is the total variation of the boundary data times the bound on the arc lengths of characteristics (see lemma \ref{Lem:ArcLenBound} b) ).
			\smallskip
			
			The total variation of the second summand is exactly the $\L^1$-norm of $\nabla u_2$:
			\begin{align*}
				\int\limits_{\Omega} \;  | \nabla u_2 (x)|  \;dx  &= \int\limits_a^b \int\limits_{0}^{1} \;  | \nabla u_2 \circ \xi \cdot \det D\xi |   \;dt \; ds \\
				& =  \int\limits_a^b \int\limits_{0}^{1}  \left| \;  -\pd_s \xi^{\perp} \; \pd_t v_2 +\pd_t \xi^{\perp} \; \pd_s v_2 \; \right|   \; dt \; ds  \\
				& \leq   \int\limits_a^b \int\limits_{0}^{1}  |\pd_s \xi| \; |\pd_t v_2| \; dt \; ds + \int\limits_a^b \int\limits_{0}^{1} |\pd_t \xi| \; |\pd_s v_2|    \; dt \; ds
			\end{align*}
			This step is completed if the last two integrals are bounded. The partial derivative of $v_2$ w.r.t. $t$ is
			\[
				\pd_t v_2 = f_0 \circ \xi = \frac{f}{\left< c, \nabla T_0\right>}\circ \xi = f \circ \xi \cdot |\partial_t \xi| \;.
			\]
			Additionally, by using relation \refEq{eqn:RelDet} from corollary \ref{Cor:PolarCoord} c), we obtain the bound
			\begin{align*}
				\int\limits_a^b & \int\limits_{0}^{1}  |\pd_s \xi| \; |\pd_t v_2| \; dt \; ds 
				\; = \; \int\limits_a^b \int\limits_{0}^{1}  |\pd_s \xi| \; |\pd_t \xi| \; |f \circ \xi| \; dt \; ds \\
				& \leq \; \frac{1}{\beta}  \int\limits_a^b \int\limits_{0}^{1}  \det D\xi \; |f \circ \xi| \; dt \; ds 
				 = \; \frac{1}{\beta}  \int\limits_{\Omega}  |f(x)|  \;dx \leq \frac{1}{\beta} \cdot \|f\|_{\infty} \cdot \Lm^2(\Omega) 
			\end{align*}
			on the first integral. The partial derivative of $v_2$ w.r.t. $s$ is
			\[
				\pd_s v_2 = \int\limits_0^t \SP< {\nabla f_0 \circ \xi(\tau,s)} , {\pd_s \xi(\tau,s)} >  \; d\tau \;.
			\] 
			With this, we estimate the second integrand by
			\[
				|\pd_t \xi| \; |\pd_s v_2| \leq \frac{1}{\beta \cdot m_0} \cdot \int\limits_0^{1} \left| \SP< {\nabla f_0 \circ \xi(\tau,s)} , {\pd_s \xi(\tau,s)} > \right| \; d\tau \;.
			\]
			Because the latter is independent of $t$, we obtain
			\[
				\int\limits_a^b \int\limits_{0}^{1} |\pd_t \xi| \; |\pd_s v_2|    \; dt \; ds \leq \frac{1}{\beta \cdot m_0} \cdot 
				\int\limits_a^b \int\limits_0^{1}  \left| \SP< {\nabla f_0 \circ \xi(\tau,s)} , {\pd_s \xi(\tau,s)} > \right| \; d\tau \; ds\;.
			\]
			According to the definition of $f_0$ in equation \refEq{eqn:DeffNull}, we expand $\nabla f_0$ to
			\[
				\nabla f_0^T = \frac{1}{\SP<{c},{\nabla T_0}>} \left(\nabla f^T - \frac{f}{\SP<{c},{N}>} \left(N^T Dc + c^T \frac{D^2 T_0}{|\nabla T_0|}\right) \right) \;.
			\]
			Moreover, with $\pd_s \xi = N^{\perp} \circ \xi \cdot |\pd_s \xi|$ and $|\pd_t \xi| = 1/\SP<{c},{\nabla T_0}> \circ \xi$, we end up with
			\begin{align*}
				\SP< {\nabla f_0 \circ \xi} , {\pd_s \xi } > & = 
				|\pd_s \xi|  |\pd_t \xi| \left( \SP< {\nabla f} , {N^{\perp}} > - \frac{f}{\SP< {c}, {N}>} \cdot \left(N^T Dc N^{\perp} + c^T \frac{D^2 T_0}{|\nabla T_0|} N^{\perp}\right) \right)\circ \xi \:.
			\end{align*}
			Finally, by writing $N$ as $N = \nabla T_0 / |\nabla T_0|$, we have $DN = N^{\perp} \cdot N^{\perp T} \cdot D^2 T_0 /|\nabla T_0| $ and can express the summand involving $D^2 T_0$
			in terms of $DN$:
			\[
				N^{\perp T} \cdot DN \cdot c = N^{\perp T} \cdot \frac{D^2 T_0}{|\nabla T_0|} \cdot c = c^T \cdot \frac{D^2 T_0}{|\nabla T_0|} \cdot N^{\perp} \;.
			\]
			Hence, there is  the integrable upper bound 
			\[
				|\SP< {\nabla f_0 \circ \xi} , {\pd_s \xi} > | \leq  \frac{\det D\xi}{\beta}  \left( \| \nabla f\|_{\infty} + \frac{\|f\|_{\infty}}{\beta} ( | Dc | +  | DN |) \right)\circ \xi \; ,
			\]
			which implies
			\begin{align*}
				\int\limits_a^b \int\limits_{0}^{1} |\pd_t \xi|&  \, |\pd_s v_2|    \, dt \, ds  \leq  
				 \frac{\| \nabla f\|_{\infty}}{\beta^2 m_0}  \Lm^2(\Omega) + \frac{\|f\|_{\infty}}{\beta^3 m_0} \left( \| Dc \|_{L^1(\Omega)} +  \| DN \|_{L^1(\Omega)}\right)  \, .
			\end{align*}
			Combining the partial results we firstly get a bound on the $\L^1$-norm of $\nabla u_2$
			\begin{align*}
				\| \nabla u_2 \|_{L^1(\Omega)} \leq & \; \left( \frac{\|f\|_{\infty}}{\beta}  +  \frac{\| \nabla f\|_{\infty}}{\beta^2 m_0} \right) \cdot \Lm^2(\Omega)  
				 + \frac{\|f\|_{\infty}}{\beta^3 m_0} \left( \| Dc \|_{L^1(\Omega)} +  \| DN \|_{L^1(\Omega)}\right) \; ,
			\end{align*}
			and secondly learn that the total variation $|Du|(\Omega \wo \Sigma)$ is bounded by $M_{\Omega \wo \Sigma}$.
			\smallskip
		\item Now, we want to view $u$ as a $\BV$ function on the domain $\Omega$. Thus, test functions stem from $\C_c^1(\Omega)^2$ and we have to study the limit $\lambda \to 1$
			in equation \refEq{eqn:DuMeas1}.  In part a) we already have bounds on the total variation of the components concerning $c_0^{\perp}(x) \cdot \mu$ and $\nabla u_2(x) \cdot \Lm^2$,
			which do not depend on $\lambda$. Hence, these bounds stay the same, and we can concentrate on the remainder
			\begin{align*}
				\int\limits_{\Omega}    \SP<  {\varphi (x)} , {-N(x) } > \,  u(x) \;d\Hm^1 \rto \lambda(x) &
				 =   \int\limits_a^b   \SP< {\psi(\lambda,s)} , {\pd_s \xi^{\perp}(\lambda,s)} > \, v(\lambda,s)  \;ds \;.
			\end{align*}
			%% Skizze
			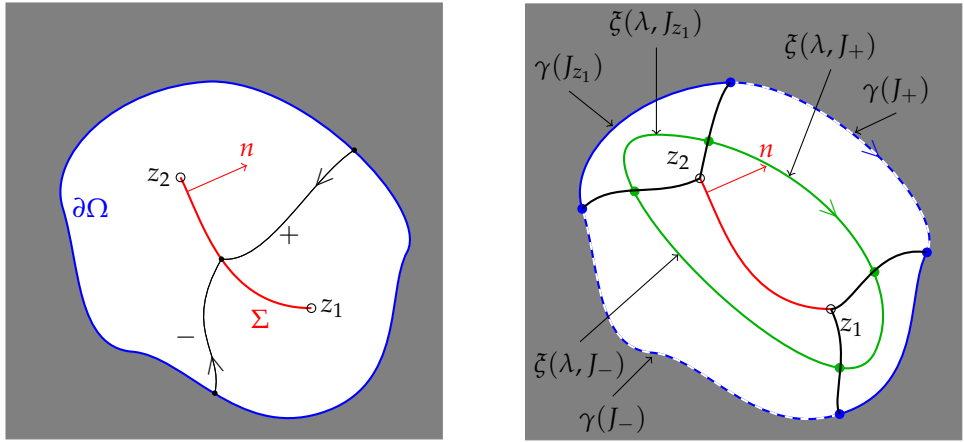
\begin{figure}[t]
			\begin{minipage}{0.47\textwidth}
			\begin{tikzpicture}[scale=5.8]
				\filldraw[gray] (0,0) rectangle (1,1);
				\filldraw[white] (0.13,0.53) .. controls (0.093,0.65) and (0.25,0.81) .. (0.47,0.82)
				 .. controls (0.69,0.82) and (0.97,0.53) .. (0.92,0.43)
				 .. controls (0.87,0.34) and (0.92,0.12) .. (0.72,0.059)
				 .. controls (0.53,0) and (0.41,0.19) .. (0.29,0.2)
				 .. controls (0.16,0.21) and (0.17,0.4) .. (0.13,0.53);
				\draw[blue,thick] (0.13,0.53) .. controls (0.093,0.65) and (0.25,0.81) .. (0.47,0.82)
				 .. controls (0.69,0.82) and (0.97,0.53) .. (0.92,0.43)
				 .. controls (0.87,0.34) and (0.92,0.12) .. (0.72,0.059)
				 .. controls (0.53,0) and (0.41,0.19) .. (0.29,0.2)
				 .. controls (0.16,0.21) and (0.17,0.4) .. (0.13,0.53) node[anchor=west]{$\pd \Omega$} ; %boundary
				
				%\draw[black] (0.4,0.6)--(0.45,0.5) -- (0.5,0.3) -- (0.7,0.3);
				\draw[red,thick] (0.4,0.6).. controls (0.45,0.5) and (0.5,0.3) ..  (0.7,0.3); %sigma 
				\draw (0.4,0.6) node[anchor=east]{$z_2$} circle (0.1mm);
				\draw (0.7,0.3) node[anchor=west]{$z_1$} circle (0.1mm);
				\draw[red] (0.625,0.275) node[anchor=east]{$\Sigma$};
				
				%char top-bd
				\draw (0.4938 , 0.4125) .. controls (0.6 , 0.4) and (0.7 , 0.6).. node[below]{$+$} (0.7982 , 0.6625); 
				\draw (0.4938 , 0.4125) .. controls (0.6 , 0.4) and (0.7 , 0.6).. node[sloped,near end]{$<$}(0.7982 , 0.6625);
				%\draw[->,green]  (0.4938 , 0.4125) -- (0.3448 ,   0.4300) node[anchor=east]{$c^+$};
				
				%char bot-bd
				\draw (0.4938 , 0.4125) .. controls (0.4 , 0.25) and (0.5 , 0.2).. node[below,anchor=east]{$-$}(0.4786 , 0.1051); 
				\draw (0.4938 , 0.4125) .. controls (0.4 , 0.25) and (0.5 , 0.2).. node[sloped,near end]{$<$}(0.4786 , 0.1051); 
				%\draw[->,green]  (0.4938 , 0.4125) -- (0.5687  ,  0.5424) node[anchor=south]{$c^-$};
				
				%Normals
				%\draw[->,red]  (0.4938 , 0.4125)--(0.6166,0.4985) node[anchor=south]{$N^-$};
				%\draw[->,red]  (0.4938 , 0.4125)--(0.3709,0.3265) node[anchor=south]{$N^+$};
				\draw[->,red]  (0.4152 , 0.5673)--(0.5524, 0.6277) node[anchor=south]{$n$};
				
				\filldraw (0.4938 , 0.4125) circle (0.05mm); %point on sigma
				\filldraw (0.7982 , 0.6625) circle (0.05mm); %point on top-bd
				\filldraw (0.4786 , 0.1051) circle (0.05mm); %point on bot-bd
			\end{tikzpicture}
			\end{minipage}
			\hfill
			\begin{minipage}{0.47\textwidth}
			\begin{tikzpicture}[scale=5.8]
				\colorlet{darkgreen}{green!70!black}
				
				\filldraw[gray] (0,0) rectangle (1,1);
				\filldraw[white] (0.13,0.53) .. controls (0.093,0.65) and (0.25,0.81) .. (0.47,0.82)
				 .. controls (0.69,0.82) and (0.97,0.53) .. (0.92,0.43)
				 .. controls (0.87,0.34) and (0.92,0.12) .. (0.72,0.059)
				 .. controls (0.53,0) and (0.41,0.19) .. (0.29,0.2)
				 .. controls (0.16,0.21) and (0.17,0.4) .. (0.13,0.53);
				\draw[blue,thick] (0.13,0.53) .. controls (0.093,0.65) and (0.25,0.81) .. (0.47,0.82)
				 .. controls (0.69,0.82) and (0.97,0.53) .. node[sloped]{$>$} (0.92,0.43)
				 .. controls (0.87,0.34) and (0.92,0.12) .. (0.72,0.059)
				 .. controls (0.53,0) and (0.41,0.19) .. (0.29,0.2)
				 .. controls (0.16,0.21) and (0.17,0.4) .. (0.13,0.53); % node[anchor=west]{$\partial \Omega$} ; %boundary

				\draw[white,thick, dashed] (0.47,0.82) .. controls (0.69,0.82) and (0.97,0.53) .. (0.92,0.43); % \gamma_+
				\filldraw[blue] (0.47,0.82) circle (0.1mm); % Endknoten \gamma_+
				\filldraw[blue] (0.92,0.43) circle (0.1mm);

				\draw[white,thick, dashed] (0.72,0.059) .. controls (0.53,0) and (0.41,0.19) .. (0.29,0.2) .. controls (0.16,0.21) and (0.17,0.4) .. (0.13,0.53); % \gamma_-
				\filldraw[blue] (0.72,0.059) circle (0.1mm); % Endknoten \gamma_-
				\filldraw[blue] (0.13,0.53) circle (0.1mm);
				
				%\draw[black] (0.4,0.6)--(0.45,0.5) -- (0.5,0.3) -- (0.7,0.3);
				\draw[red,thick] (0.4,0.6).. controls (0.45,0.5) and (0.5,0.3) .. (0.7,0.3); %sigma
				\draw (0.4,0.6) node[anchor=south east]{$z_2$} circle (0.1mm);
				\draw (0.7,0.3) node[anchor=north west]{$z_1$} circle (0.1mm);

				% \lambda-level
				\draw[color= darkgreen,thick] (0.3 , 0.7).. controls (0.6,0.7) and (0.9,0.4) .. node[sloped]{$>$} (0.8 , 0.2);
				\draw[color= darkgreen,thick] (0.3 , 0.7)  .. controls (0,0.7) and (0.7,0) .. (0.8 , 0.2);
												
				\filldraw[color= darkgreen] (0.42,0.685) circle (0.1mm); % Endknoten \xi_+
				\filldraw[color= darkgreen] (0.8,0.385) circle (0.1mm);
				
				\filldraw[color= darkgreen] (0.25,0.57) circle (0.1mm); % Endknoten \xi_-
				\filldraw[color= darkgreen] (0.72,0.165) circle (0.1mm);
				
				\draw[->] (0.3,0.9) node[anchor=south]{$\xi(\lambda,J_{z_1})$} --(0.3 , 0.7); %pointer auf xi_z1
				\draw[->] (0.1,0.8) node[anchor=south]{$\gamma(J_{z_1})$} --(0.2 , 0.72); %pointer auf gamma_z1
				
				\draw[->] (0.85,0.75) node[anchor=south]{$\gamma(J_{+})$} --(0.755 , 0.705); %pointer auf gamma_+
				\draw[->] (0.7,0.85) node[anchor=south]{$\xi(\lambda, J_{+})$} --(0.605, 0.605); %pointer auf xi_+
				
				\draw[->] (0.2,0.1) node[anchor=north]{$\gamma(J_{-})$} --(0.29 , 0.19); %pointer auf gamma_-
				\draw[->] (0.125,0.225) node[anchor=north]{$\xi(\lambda,J_{-})$} --(0.35 , 0.425); %pointer auf xi_-
				
				% hit at z_1
				\draw[thick] (0.4,0.6) .. controls (0.3 , 0.55) and (0.2 , 0.6).. (0.13,0.53);  
				\draw[thick] (0.4,0.6) .. controls (0.43 , 0.75) and (0.45 , 0.8).. (0.47,0.82);
								
				% hit at z_2
				\draw[thick] (0.7,0.3) .. controls (0.75 , 0.2) and (0.7 , 0.1).. (0.72,0.059);
				\draw[thick] (0.7,0.3) .. controls (0.75 , 0.3) and (0.8 , 0.45).. (0.92,0.43);
				
				\draw[->,red]  (0.4152 , 0.5673)--(0.5524, 0.6277) node[anchor=south]{$n$};  % normale Sigma

			\end{tikzpicture}
			\end{minipage}
			\caption{Decomposition of the level line $\{T_0 = \lambda\}=\xi(\lambda,\dotarg)$ into 4 parts.
			The left image shows the domain $\Omega$ in white, the boundary $\bd \Omega$ in blue, and in red the stop set $\Sigma$ with terminal nodes $z_1$ and $z_2$ .
			The red vector $n$ is the normal of $\Sigma$ and defines the plus and the minus side. The black curves are two characteristics, both start at $\bd \Omega$, but
			the upper one approaches $\Sigma$ from the plus side while the other comes from the minus side. 
			The right image illustrates the decomposition of $\xi(\lambda,\dotarg)$, the green curve. In black we have characteristics which hit the terminal nodes. For
			each node only the ''first'' and ''last'' one of the characteristics hitting the node are painted.
			In solid blue we have the parts $\gamma(J_{z_1})$ and $\gamma(J_{z_2})$ which comprise all the start points for which characteristics hit terminal nodes of $\Sigma$.
			The complement is painted dashed blue and consists of the two parts $\gamma(J_+)$ and $\gamma(J_-)$. Here, $\gamma(J_+)$ ($\gamma(J_-)$) contains all start points for which characteristics
			approach $\Sigma$ from the plus side (minus side) and end in the relative interior of $\Sigma$. With parameter sets $J_{z_1}$, $J_{z_2}$, $J_+$, and $J_-$,
			that correspond to the decomposition of the boundary curve $\gamma$, we obtain
			an analogous decomposition of the level line $\{T_0 = \lambda\}=\xi(\lambda,\dotarg)$ by $\xi(\lambda,J_{z_1})$, $\xi(\lambda,J_{z_2})$, $\xi(\lambda,J_+)$, $\xi(\lambda,J_-)$.
			If now $\lambda \to 1$, i.e., $\{T_0 = \lambda\} \to \Sigma$, the curves $\xi(1,J_+)$ and $\xi(1,J_-)$ remain and are two parametrizations of $\Sigma$.
			} \label{Fig:PolarCoord}
			\end{figure}
			
			In corollary \ref{Cor:PolarCoord} d) we introduced a partition of $I$, the domain of $\gamma$, such that $\xi(1,.)|_{J_{k,+}}$ and $\xi(1,.)|_{J_{k,-}}$
			are both regular parametrizations of $\inner{\Sigma}_k$, one for the plus side and one for the minus side of $\inner{\Sigma}_k$, i.e.,
			\[
				I = \bigcup\limits_{k=1}^n \left( J_{k,+} \cup J_{k,-} \right) \; \cup \; J \;.
			\]
			The remaining part $J$ comprises all $s$ for which the characteristic $\xi(\lambda,s)$ hits a singular node of $\Sigma$, i.e., a
			terminal-, branching- or kink node, as $\lambda$ tends to 1. As before, let $z_1,\ldots,z_m$ denote the singular nodes of $\Sigma$. 
			Then we partition $J$ into a collection $J_{z_1},\ldots, J_{z_m}$ by:
			\[
				s \in J_{z_l} \quad :\Leftrightarrow \quad \lim\limits_{\lambda \to 1} \xi(\lambda,s) = z_l \;,\qquad J = \bigcup\limits_{l=1}^{m} J_{z_l} \;.
			\]
			For an illustration of this decomposition, see figure \ref{Fig:PolarCoord}. 
			
			By these partitions, we decompose the integral
			\begin{align*}
				\int\limits_a^b  & \SP< {\psi(\lambda,s)} , {\pd_s \xi^{\perp}(\lambda,s)} > \, v(\lambda,s)  \;ds 
				= \; \sum\limits_{l=1}^{m}  \int\limits_{J_{z_l}}  \ldots \; ds 
				 \; + \sum\limits_{k=1}^n \left( \; \int\limits_{J_{k,+}}   \ldots \; ds + \int\limits_{J_{k,-}}  \ldots \; ds \right) \;.
			\end{align*}
			For the $J_{z_l}$-summands, we have the estimate
			\begin{align*}
				& \left| \; \int\limits_{J_{z_l}}  \SP< {\psi(\lambda,s)} , {\pd_s \xi^{\perp}(\lambda,s)} > \, v(\lambda,s)  \;ds \right|
				= \left| \; \int\limits_{\xi(\lambda,J_{z_l} )}    (\SP<  {\varphi } , {-N } > \cdot  u)(x) \;d\Hm^1(x) \right| \\
				& \qquad \leq \|\varphi\|_{\infty} \cdot \|u\|_{\L^{\infty}(\Omega)} \cdot \Hm^1(\; \xi(\lambda,J_{z_l} )\;) \;.
			\end{align*}
			Because the arc $\xi(\lambda,J_{z_l})$ degenerates to the single point $z_l$ (see figure \ref{Fig:PolarCoord}), the right-hand side becomes zero in the limit
			\[
				\lim\limits_{\lambda \to 1} \Hm^1(\; \xi(\lambda,J_{z_l} )\;) = 0 \;, 
			\]
			and the contribution of those summands vanishes.
			
			For the remaining summands we perform only those which go along $\xi(\lambda, J_{k,+})$, since the same argumentation applies to the others
			which go along $\xi(\lambda, J_{k,-})$. For $\xi(1,.):J_{k,+} \to \inner{\Sigma}_k$ the tangent is given by the limit
			\[ 
				\pd_s \xi(1,s) = \lim\limits_{\lambda \to 1} \pd_s \xi(\lambda,s)  \qquad s \in J_{k,+} \;.
			\]
			By requirement \ref{Req:TimeFuncSigma} part 2, the field of normals extends to $\inner{\Sigma}_k$. This means
			\[
				\lim\limits_{\lambda \to1} -N \circ \xi(\lambda,s) = n_k \circ \xi(1,s) \qquad s \in J_{k,+} \;.
			\]
			Finally, the second component $s(x)$ of $\xi^{-1}(x)$ (see corollary \ref{Cor:PolarCoord} b) ) extends one-sided onto $\inner{\Sigma}_k$, and so does $u$.
			Let $z \in \inner{\Sigma}_k$, with corresponding $s^+(z) \in J_{k,+}$. Then we set
			\begin{align}\label{eqn:ExtendSol}
				u_k^+&(z) := \lim\limits_{\lambda \to 1} v(\lambda,s^+(z)) = v(1,s^+(z))  \;.
			\end{align}
			Conversely, the following relation holds,
			\begin{align*}
				u_k^+&\circ \xi(1,s) = \lim\limits_{\lambda \to 1} u \circ \xi(\lambda,s) =  v(1,s) \;, \qquad s \in J_{k,+} \;.
			\end{align*}		
			Now, we can turn to the limit. For abbreviation let
			\[
				h(\lambda,s) := (\SP< {\varphi}, {-N} > \cdot u) \circ \xi(\lambda,s) \;.
			\]
			Then we have
			\begin{align*}
			& \left| \; \int\limits_{\xi(\lambda,J_{k,+})} (\SP<  {\varphi } , {-N } > \cdot  u)(x) \; d\Hm^1(x) 
			- \int\limits_{\Omega} (\SP<  {\varphi } , {n_k } > \cdot  u_k^+ )(x) \; d\Hm^1\rto \Sigma_k(x) \right| \\
			& = \; \left| \; \int\limits_{J_{k,+}} h(\lambda,s) \cdot |\pd_s \xi(\lambda,s)| \;ds - \int\limits_{J_{k,+}} h(1,s) \cdot |\pd_s \xi(1,s)| \;ds \right| \\
			%& \leq \; \int\limits_{J_{k,+}}  |h(1,s) - h(\lambda,s)| \cdot  |\pd_s \xi(1,s)|  +  ||\pd_s \xi(1,s)| - |\pd_s \xi(\lambda,s)|| \cdot |h(\lambda,s)| \; ds \\
			& \leq \; \int\limits_{J_{k,+}}  |h(1,s) - h(\lambda,s)| \cdot  |\pd_s \xi(1,s)| \; ds \\
			& \qquad\quad 
			+ \quad \|h \circ \xi^{-1}\|_{\L^{\infty}(\Omega)} \int\limits_{J_{k,+}}   ||\pd_s \xi(1,s)| - |\partial_s \xi(\lambda,s)||  \; ds \;.
			\end{align*}	
			By the extensions of $N$ and $u$ the product $|h(1,s) - h(\lambda,s)| \cdot  |\pd_s \xi(1,s)|$ tends to zero for  every $s \in  J_{k,+}$. Furthermore, it has the integrable bound
			\[
				|h(1,s) - h(\lambda,s)| \cdot  |\pd_s \xi(1,s)| \leq 2\|h \circ \xi^{-1}\|_{\L^{\infty}(\Omega)} \cdot  |\pd_s \xi(1,s)| \;.
			\]
			Thus, by dominated convergence, the corresponding integral vanishes as $\lambda \to 1$. By the same argument the second integral tends to zero, too.
			\smallskip
			
			Summarizing, this means 
			\begin{align*}
				& \int\limits_{\xi(\lambda,J_{k,+})} \SP<  {\varphi } , {-N } > \cdot  u \; d\Hm^1(x)
				\to \int\limits_{\Omega} \SP<  {\varphi } , {n_k } > \cdot  u_k^+  \; d\Hm^1\rto \Sigma_k(x)  \; ,\\
				& \int\limits_{\xi(\lambda,J_{k,-})} \SP<  {\varphi } , {-N } > \cdot  u \; d\Hm^1(x)
				\to \int\limits_{\Omega} \SP<  {\varphi } , {-n_k } > \cdot  u_k^-  \; d\Hm^1\rto \Sigma_k(x)  \; ,
			\end{align*}
			as $\lambda \to 1$, and together we obtain the jump part for $\Sigma_k$
			\begin{align*}
				\int\limits_{\xi(\lambda,J_{k,+}) \cup \xi(\lambda,J_{k,-})} & \SP<  {\varphi } , {-N } > \cdot  u  \; d\Hm^1(x) \;  \to %\\ 
				%& 
				\int\limits_{\Omega} (u_k^+-u_k^-) \cdot \SP<  {\varphi } , {n_k } > \; d\Hm^1\rto \Sigma_k(x) \;.
			\end{align*}
			The bound on the total variation of the jump part is
			\[
				\left| \; \int\limits_{\Omega} (u_k^+-u_k^-) \; \SP<  {\varphi } , {n_k } > \; d\Hm^1\rto \Sigma_k(x)  \right| \leq 2 \cdot \|u\|_{\L^{\infty}(\Omega)}  \cdot \Hm^1(\Sigma_k) \;,
			\]
			whereas $\|\varphi\|_{\infty} \leq 1$.
			What remains to show is that the one-sided limits $u_k^+,u_k^-$ defined above are in fact the $\BV$-traces
			$u_{\Sigma_k}^+$ and $u_{\Sigma_k}^-$ of $u$ on $\Sigma_k$. This is true, but we postpone this point to lemma \ref{Lem:SRS}.
					
			Clearly, the complete additional jump part is given by 
			\[
				\sum\limits_{k=1}^n \; \int\limits_{\Omega} (u_k^+(x)-u_k^-(x)) \; \SP<  {\varphi } , {n_k } >(x) \; d\Hm^1\rto \Sigma_k(x)
			\]
			with $2\|u\|_{\L^{\infty}(\Omega)} \cdot \Hm^1(\Sigma)$ as a bound on its total variation.
			We have shown that $u$ is in fact an element of $\BV(\Omega)$ with $\|u\|_{\BV(\Omega)}$ bounded by the given data, and its derivative measure reads
			\[
				Du = \sum\limits_{k=1}^n (u_k^+(x)-u_k^-(x))\; n_k(x) \cdot \Hm^1 \rto \Sigma_k \; + \; c_0^{\perp}(x) \cdot \mu \; + \; \nabla u_2(x) \cdot \Lm^2 \;.
			\]
	\end{enumerate}
	\hfill
\end{proof}
\medskip

Theorem \ref{Theo:ElemBV} part a) shows that the candidate solution $u$ belongs to $\BV(\Omega \wo \Sigma)$ and the derivative $Du$,
on $\Omega \wo \Sigma$,  is given by $Du = c_0^{\perp}(x) \cdot \mu \; + \; \nabla u_2(x) \cdot \Lm^2$.
From orthogonality, we infer 
\begin{align*}
	\SP< {c_0(x)}, {Du} > &= \SP< {c_0(x)},{\nabla u_2(x)}> \cdot \Lm^2 \;, \quad \mbox{in} \quad \Omega \wo \Sigma
\end{align*}
and as in the classical case, one checks that $\SP< {c_0(x)},{\nabla u_2(x)}> = f_0(x)$; the relation 
between backward and forward characteristics according to equation \refEq{eqn:RelFBChar} is $\xi(T_0(x)-t,s(x)) = \eta(t,x)$. Now, substitute $x = \xi(\tau,s)$ to get
$\xi(\tau-t,s) = \eta(t,\xi(\tau,s))$, differentiate w.r.t. $\tau$, resubstitute and obtain the relation
\begin{equation}\label{eqn:BCharDiff}
	-\eta'(t,x) = D_x \eta(t,x) \cdot c_0(x) \;.
\end{equation}
The description \refEq{eqn:SolOrig} of $u_2$, equation \refEq{eqn:BCharDiff}, and the equality $\SP< {c_0} , {\nabla T_0} > = 1$ imply
\begin{align*}
	\SP< {c_0(x)}, {\nabla u_2(x)} > &= f_0 \circ \eta(T_0(x),x) - \int\limits_0^{T_0(x)} \SP< {\nabla f_0 \circ \eta(\tau,x)}, {\eta'(\tau,x)} > \; d\tau = f_0(x) \;.
\end{align*}
Consequently, the candidate $u$ solves the PDE of problem \refEq{eqn:LinProblem}.

In the next lemma we study the traces of the candidate $u$ along level lines of $T_0$ and observe thereby that $u$ satisfies the
Dirichlet boundary condition of problem \refEq{eqn:LinProblem} as $BV$-trace.
\medskip

\begin{lem} \label{Lem:SRS}
	(Start / restart / stop)
	\begin{enumerate}[a)]
		\item (start): $u$ satisfies the boundary condition $u|_{\bd \Omega} = u_0$ as $\BV$ boundary trace, i.e.,
			\[ 
				\lim\limits_{r \to 0+} \frac{1}{r^2} \int\limits_{\Omega \cap B_r(z)} |u_0(z) - u(x)| \; dx = 0
			\]
			for every Lebesgue point $z \in \bd \Omega$ of $u_0$.
			\medskip
		\item (restart): for every $z \in \Omega \wo \Sigma$ that corresponds to a Lebesgue point $z'$ of $u_0$, that means $z'=\gamma(s(z))$
			is a Lebesgue point of $u_0$, we have
			\begin{align*}
				\lim\limits_{r \to 0+} \frac{1}{r^2} \int\limits_{B^>_r(z)} |u(z) - u(x)| \; dx &= 0 \; & \mbox{and} \quad
				\lim\limits_{r \to 0+} \frac{1}{r^2} \int\limits_{B^<_r(z)} |u(z) - u(x)| \; dx &= 0 \;.
			\end{align*}
			Here, $B^<_r(z)$ and $B^>_r(z)$ -- for $r$ small enough -- denote the cut-off disks
			\begin{align*}
				B^<_r(z) &:= \{x \in B_r(z) : T_0(x) < T_0(z)\} \;, &
				B^>_r(z) &:= \{x \in B_r(z) : T_0(x) > T_0(z)\} \;.
			\end{align*}
			\smallskip
			
			Let $\Gamma := \chi_{T_0=\lambda}$ a $\lambda$-level of $T_0$ for some $0 < \lambda <1$ and let $\Gamma$ be oriented by $N|_{\Gamma}$.
			The result above means that the traces $u_{\Gamma}^+$ and $u_{\Gamma}^-$ are identical, thus the restriction $u|_{\Gamma}$ is well-defined.
			Moreover, we have $u|_{\Gamma} \in \BV(\Gamma)$.
			\medskip
		\item (stop): for every $z \in \inner{\Sigma}_k$ that corresponds to a Lebesgue point $z'$ of $u_0$ w.r.t. the plus side of $\inner{\Sigma}_k$, that means $z'=\gamma(s^+(z))$
			is a Lebesgue point of $u_0$,  we have
			\[ 
				\lim\limits_{r \to 0+} \frac{1}{r^2} \int\limits_{B^+_r(z)} |u^+_k(z) - u(x)| \; dx = 0 \;.
			\]
			Here, $u^+_k(z)$ is defined by equation \refEq{eqn:ExtendSol} and $B^+_r(z)$ -- for $r$ small enough -- is the cut-off disk centered at $z$, restricted to the plus side.		
			Hence, the trace $u_{\Sigma_k}^+$ is given by $u^+_k$. Moreover, we have $u^+_k \in \BV(\inner{\Sigma}_k)$.			
			The analogous result holds true w.r.t. the minus side.
	\end{enumerate}
\end{lem}
\medskip

\begin{proof} 
	Let $z \in \cl{\Omega} \wo \Sigma$ and let $(\tau,\sigma)$ be its characteristic coordinates, i.e., $z = \xi(\tau,\sigma)$.
	First, we discuss part b) and in particular the case of $B^>_r(z)$. Choose $r > 0$ suitably small. By changing variables we get
	\begin{align*}
		\frac{1}{r^2} \int\limits_{B^>_r(z)}  |u(z) - u(x)| \; dx 
		 &= \frac{1}{r^2} \int\limits_{s_-(r)}^{s_+(r)} \int\limits_ \tau ^{t_+(r,s)} |v(\tau,\sigma) - v(t,s)| \det D\xi(t,s)\; dt \; ds \;.
	\end{align*}
	The transformed integrand can be estimated by 
	\[
		|v(t,s) - v(\tau,\sigma)| \leq |\gamma^*u_0(s) - \gamma^*u_0(\sigma)|  +  \O(s-\sigma) + \O(t-\tau)\;,
	\]
	because for the parts $v_2$ of $v$ which concern the right-hand side $f$, we have
	\begin{align*}
		& |v_2(t,s) - v_2(\tau,\sigma)| = \left| \int\limits_0^t f_0 \circ \xi (h,s) \; dh - \int\limits_0^{\tau} f_0 \circ \xi (h,\sigma) \; dh \right| \leq \\
		& \int\limits_0^t |\nabla f_0 \circ \xi (h,s_*)|\, |\pd_s\xi (h,s_*)| \, |s-\sigma| \; dh + \left| \int\limits_{\tau}^t f_0 \circ \xi (h,\sigma) \; dh \right| = \O(s-\sigma) + \O(t-\tau) .
	\end{align*}
	As the determinant $\det D\xi$ is continuous and positive, we estimate it by a constant.
	Finally, any level line of $T_0$ is a regular $\C^1$ curve and so the cut-off disk $B^>_r(z)$ tends to a half disk, oriented along the tangent $\pd_s \xi(\tau,\sigma)$ of this curve.
	This implies the asymptotic formulas  
	\begin{align*}
		s_-(r) &= \sigma - \O(r) & s_+(r) &= \sigma + \O(r) & t_+(r,s) &= \tau + \O(r) \;.
	\end{align*}
	Armed with the last considerations, we have
	\begin{align*}
		\frac{1}{r^2} \int\limits_{B^>_r(z)}  |u(z) - u(x)| \; dx 
		& \leq \; \frac{C}{r} \int\limits_{s_-(r)}^{s_+(r)} |\gamma^*u_0(s) - \gamma^*u_0(\sigma)| \; ds  + \O(r) \; .
	\end{align*}
	The latter expression, as $r \to 0$, tends to zero for any Lebesgue-point $\sigma$ of $\gamma^*u_0$, i.e., for any Lebesgue-point $z' = \gamma(\sigma)$ of $u_0$.
		
	So far, we have got the assertion for the case  of $B^>_r(z)$. In the case  of $B^<_r(z)$
	the one and only difference is that, after having changed variables, 
	we have to integrate the $t$-variable over the interval $\IV[{t_-(r,s)},\tau]$ and to perform the same steps as above. Together, this proves the first statement of part b).
		
	For part a), we have $z \in \bd \Omega$ and hence $\Omega \cap B_r(z) = B^>_r(z)$. So, the argument used above shows that $u$ satisfies the boundary condition as $\BV$-trace.
			
	The first statement of part b) implies that the restriction $u|_{\Gamma}$ is well-defined. When parametrizing $\Gamma$ regularly by $\xi_{\lambda}(s) := \xi(\lambda,s)$, we obtain 
	\begin{align*}
		\xi_{\lambda}^*u(s) &= u \circ \xi(\lambda,s) = v(\lambda,s) = \gamma^*u_0(s) + v_2(\lambda,s)  & s &\; \in \R \;.
	\end{align*}
	Since $v_2(\lambda,.)$ is a periodic $\C^1$ function, while $\gamma^*u_0$ is a periodic $\BV$ function, the sum  $\xi_{\lambda}^*u$ 
	is a periodic $\BV$ function, and  consequently $u|_{\Gamma} \in \BV(\Gamma)$. This proves the second statement of part b).
		
	Finally, for part c), $z$ belongs to the stop set, $z \in \inner{\Sigma}_k$. 
	In this case, we cannot use $\xi$ because its Jacobian is singular on the stop set. But, as $\xi(1,s)$ for $s \in J_{k,+}$ regularly parametrizes $\inner{\Sigma}_k$
	(see corollary \ref{Cor:PolarCoord} d)), one can use the local diffeomorphism induced by
	\begin{align*}
		y' &= -c(y) & y(0,s) &= \xi(1,s) \;, \;  s \in J_{k,+} \;.
	\end{align*}
	Then, the same considerations as before yield the assertion. \hfill
\end{proof}
\medskip

Remark: part b) of lemma \ref{Lem:SRS} is called ''restart'' because having stopped the characteristics at some intermediate $\lambda$-level $\Gamma$ of $T_0$, 
the restarted problem
\begin{align*}
	\SP< {c(x)}, {Dw} > &= f(x) \cdot \Lm^2 \quad \mbox{ in } \quad \chi_{T_0 > \lambda} \wo \Sigma  \;,\\
	w|_{ \Gamma} &= u|_{\Gamma} \;.
\end{align*}
is of the same type as the initial one. This is, because lemma \ref{Lem:SRS} b) ensures that $u|_{\Gamma} \in \BV(\Gamma)$.
Moreover, when applying the method of characteristics here, then $w$ reproduces $u$ by $w = u|_{\chi_{T_0 > \lambda}}$.

\section{Uniqueness and Continuous Dependence}\label{Sect:Stab}
We have shown in the previous section that problem \refEq{eqn:LinProblem} has a solution in $\BV(\Omega)$. Its uniqueness can be derived as in the $\C^1$ case:
by linearity, the solution is unique if the homogenous problem
\[
	\SP< {c(x)}, {Du}> = 0 \;, \quad u|_{\bd \Omega} = 0
\] 
has a unique solution. With $v(t,s) = u \circ \xi(t,s)$ the homogeneous problem in characteristic variables reads as 
\[
	D_t v = 0 \;, \quad v(0,s) = 0 \;.
\]
Consequently, $v$ must be constant w.r.t. $t$ and is thus determined by the boundary condition. The latter implies that the unique solution of the homogeneous problem is zero. 

As the unique solution $u$, given by equation \refEq{eqn:SolOrig}, depends linearly on the boundary data $u_0$ and the right-hand side$f$, 
the continuous dependence w.r.t. $u_0$ and $f$ is already contained in theorem \ref{Theo:ElemBV}.
What remains to be shown is the continuous dependence of the solution on the transport field $c$ and the time function $T$.
\medskip

\begin{theo} \label{Theo:ContDepend}
	(Continuous dependence on $c$ and $T$)
	Let $(c_n)_{n \in \N}$ be a sequence of transport fields and let $c$ be a transport field. Let $(T_n)_{n \in \N}$ be a sequence of time functions and let $T$ be a time function.
	Then, consider the family of linear problems
	\begin{align*}
		\SP< {c_n(x)}, {Du_n} > &= f(x) \cdot \Lm^2 \;, & u_n|_{\bd \Omega} &= u_0 \;, &  \SP< {c_n(x)}, {T_n(x)} > &\geq \beta \cdot |\nabla T_n(x)| \;,\\
		\SP<{c(x)} ,{Du}> &= f(x) \cdot \Lm^2 \;, & u|_{\bd \Omega} &= u_0 \;, &  \SP< {c(x)}, {T(x)} > &\geq \beta \cdot |\nabla T(x)| \;,
	\end{align*}
	in $\Omega \wo \Sigma$. If
	\medskip
	
	\begin{enumerate}[1.] 
	\item	the sequence $(c_n)_{n \in \N}$ converges uniformly to $c$: $\|c_n - c\|_{\infty} \to 0$,
		\medskip
	\item the sequence of derivatives $(Dc_n)_{n \in \N}$ is bounded: $\|Dc_n\|_{\L^1(\Omega)}  \leq M_1$ for all $n \in \N$,
		\medskip
	\item the sequence $(T_n)_{n \in \N}$ converges uniformly to $T$: $\|T_n - T\|_{\infty} \to 0$, 
		\medskip
	\item the sequence $(\nabla T_n)_{n \in \N}$ converges uniformly to $\nabla T$: $\| \nabla T_n - \nabla T\|_{\infty} \to 0$, 
		\medskip
	\item the sequence of derivatives $(DN_n)_{n \in \N}$ is bounded: $\|DN_n\|_{\L^1(\Omega)}  \leq M_2$ for all $n \in \N$, ($N_n$ is the field of normals corresponding to $T_n$)
		\medskip
	\end{enumerate}
	then the sequence of solutions $u_n$ converges $\BV$ weakly* to $u$:
	\[
		u_n \weaksto u \quad \mbox{in} \quad \BV(\Omega) \;, \quad \mbox{as} \quad n \to \infty \; .
	\]
\end{theo}
\begin{proof} To begin with, we show the $\L^1$ convergence of $u_n$ to $u$. Since the solution linearly depends on the right-hand side $f$, it is not restrictive to set $f = 0$.
	Again, we refer to the scaled PDEs
	\begin{align*}
		\SP< {c_{n,0}(x)}, {Du_n} > &= 0  & u_n|_{\bd \Omega} &= u_0 \;, \\
		\SP<{c_0(x)} ,{Du}> &= 0  & u|_{\bd \Omega} &= u_0 \;,
	\end{align*}
	the first PDE scaled with $1/\SP<{c_n},{\nabla T_{n,0}} >$, the second one scaled with $1/\SP<{c},{\nabla T_0} >$. Thereby $T_{n,0}$ and $T_0$ denote the transformed versions of $T_n$ and $T$
	according to equation \refEq{eqn:TrafoT}.
	
	In the following, $\eta_n$ denote the backward characteristics associated with the transport fields $c_{n,0}$.
	Likewise, $\eta$ denotes the characteristics corresponding to $c_0$. In accordance with equation \refEq{eqn:SolOrig}, the solutions are
	\begin{align*}
		u_n(x) &= u_0( \eta_n(T_{n,0}(x),x)) \;,& u(x) &= u_0( \eta(T_0(x),x))  \;.
	\end{align*}
	
	We split the proof into seven steps. First step: we consider a lower level set $\Omega_{\lambda} := \{ x \in \Omega : T_0(x) \leq \lambda\}$ of $T_0$ such that 
	$2 \|u_0\|_{\L^{\infty}(\bd \Omega)} \cdot \Lm^2(\Omega \wo \Omega_\lambda) \leq \varepsilon$. Then, we can restrict the discussion to $\Omega_\lambda$:
	\[
		\|u_n - u\|_{\L^1(\Omega)} \leq \|u_n - u\|_{\L^1(\Omega_\lambda)} + 2 \|u_0\|_{\L^{\infty}(\bd \Omega)} \cdot \Lm^2(\Omega \wo \Omega_\lambda) \leq \|u_n - u\|_{\L^1(\Omega_\lambda)} + \varepsilon \;.
	\]
	
	Second step: we assume temporarily that the boundary data is continuous. With
	\[
		\|u_n - u\|_{\L^1(\Omega_\lambda)} = \int\limits_{\Omega_\lambda} |u_0( \eta_n(T_{n,0}(x),x)) - u_0( \eta(T_0(x),x))| \; dx \;,
	\]
	we will get the assertion by dominated convergence if $\eta_n(T_{n,0}(x),x) \to \eta(T_0(x),x)$ pointwise for $x \in \Omega_\lambda$.
	
	Third step: we choose $n_0$ such that for all $n \geq n_0$ the stop sets $\Sigma_n$ -- which correspond to the time functions $T_n$ -- are all compactly contained in $\Omega \wo \Omega_\lambda$
	which is possible because $T_n \to T$ uniformly. Now, we are sure that, for $n \geq n_0$ and $x \in \Omega_\lambda$, the point $x$ cannot belong to some singular set $\Sigma_n$ 
	and thus all curves $\eta_n(\dotarg,x)$ are well-defined.
	
	Fourth step: we need to estimate the difference $\eta_n(\dotarg ,x) - \eta(\dotarg ,x)$ of the backward characteristics.
	Let $x \in \Omega_\lambda$. The derivative of the difference $\eta_n(\dotarg ,x) - \eta(\dotarg ,x)$ obviously satisfies the IVP
	\[
		(\eta_n - \eta)' = c_0(\eta) - c_{n,0}(\eta_n)  \quad,\quad (\eta_n - \eta)(0 , x) = 0 \;,
	\]
	and hence, integration yields
	\[
		(\eta_n - \eta)(t,x) = \int\limits_0^t  c_0(\eta(\tau , x)) - c_{n,0}(\eta_n(\tau , x)) \; d\tau \;.
	\]
	The latter equation is valid for $t \leq \min \{T_0(x) , T_{n,0}(x)\}$. The first estimate is then
	\begin{align*}
		|\eta_n - \eta|(t,x) \leq & \int\limits_0^t  |c_0(\eta(\tau , x)) - c_0(\eta_n(\tau , x))| 
		+   |c_0(\eta_n(\tau , x)) - c_{n,0}(\eta_n(\tau , x))| \; d\tau \;.
	\end{align*}
	According to requirement \ref{Req:TransportField} 3a) we get a bound $|Dc(x)| \leq M_{\lambda}$ on $\Omega_\lambda$.
	And a similar bound $|Dc_0(x)| \leq M'_{\lambda}$ holds for the transformed transport field $c_0$, which implies the existence of a Lipschitz constant $L_\lambda$
	on $\cl{\Omega}_\lambda$. For $0 \leq t \leq \min \{T_0(x) , T_{n,0}(x)\}$ all the points $\eta(t,x)$ and $\eta_n(t,x)$ are located in $\cl{\Omega}_\lambda$ and so the next estimate is
	\[
		|\eta_n - \eta|(t,x) \leq \int\limits_0^t  L_{\lambda} \cdot |\eta_n - \eta|(\tau,x) \; d\tau + t \|c_0 - c_{n,0} \|_{\L^\infty(\Omega_{\lambda})} \; ,
	\]
	and thus Gronwall's lemma (see \cite{Walter:DiffInEq}) yields
	\begin{equation}\label{eqn:Gron}
		\begin{aligned}
			|\eta_n - \eta|(t,x) \leq  t \cdot e^{t \cdot L_{\lambda}} \cdot & \|c_0 - c_{n,0} \|_{\L^\infty(\Omega_{\lambda})} 
			\leq C_{\lambda} \cdot \|c_0 - c_{n,0} \|_{\L^\infty(\Omega_{\lambda})} \;.
		\end{aligned}
	\end{equation}
	Because of $t \leq \min \{T_0(x) , T_{n,0}(x)\} \leq T_0(x) \leq \lambda$ for $x \in \Omega_\lambda$, the constant is $C_{\lambda} = \lambda \cdot e^{\lambda \cdot L_{\lambda}}$.
	Since we have three times uniform convergence, i.e., $T_n \to T$, $\nabla T_n \to \nabla T$,  and $c_n \to c$ we obtain uniform convergence of the transformed versions
	$T_{n,0} \to T_0$, $\nabla T_{n,0} \to \nabla T_0$,  and $c_{n,0} \to c_0$ at least on $\Omega_\lambda$. Consequently, we can make the difference $|\eta_n - \eta|(t,x)$ arbitrarily small.
		
	Fifth step: we estimate the difference $|\eta_n(T_{n,0}(x),x) - \eta(T_0(x),x)|$. For abbreviation we set $\tau = T_0(x)$ and $\tau_n = T_{n,0}(x)$. We want to reuse estimate \refEq{eqn:Gron},
	but this estimate is valid only for $t \leq \min \{\tau , \tau_n\}$. So, in the case of $\tau =\min \{\tau , \tau_n\}$ we proceed as
	\begin{align*}
		|\eta_n(\tau_n,x) - \eta(\tau,x)| & \leq |\eta_n(\tau_n,x) - \eta_n(\tau,x)| + |\eta_n(\tau,x) - \eta(\tau,x)| \\
		& \leq \|-c_{n,0} \circ \eta_n(\dotarg , x)\|_{\L^{\infty}([\tau,\tau_n])} \cdot |\tau-\tau_n| +  C_{\lambda} \cdot \|c_0 - c_{n,0} \|_{\L^\infty(\Omega_{\lambda})} \\
		& \leq \frac{|\tau-\tau_n|}{\beta \cdot m_{n,0}} + C_{\lambda} \cdot \|c_0 - c_{n,0} \|_{\L^\infty(\Omega_{\lambda})} \;,
	\end{align*}
	and in the case of $\tau_n =\min \{\tau , \tau_n\}$ we supplement the other way round
	\begin{align*}
		|\eta_n(\tau_n,x) - \eta(\tau,x)| & \leq |\eta_n(\tau_n,x) - \eta(\tau_n,x)| + |\eta(\tau_n,x) - \eta(\tau,x)| \\
		& \leq C_{\lambda} \cdot \|c_0 - c_{n,0} \|_{\L^\infty(\Omega_{\lambda})} + \|-c_{0} \circ \eta(\dotarg , x)\|_{\L^{\infty}([\tau_n,\tau])} \cdot |\tau-\tau_n| \\
		& \leq C_{\lambda} \cdot \|c_0 - c_{n,0} \|_{\L^\infty(\Omega_{\lambda})} + \frac{|\tau-\tau_n|}{\beta \cdot m_{0}} \;.
	\end{align*}
	For each $n$, the value $m_{n,0} > 0$ denotes the minimum of $|\nabla T_{n,0}|$ according to lemma \ref{Lem:ArcLenBound} part a), 
	and because of $|\nabla T_{n,0}| \to |\nabla T_0|$ locally uniformly we also have $m_{n,0} \to m_0$.
	So, we have everything to conclude  $|\eta_n(T_{n,0}(x),x) - \eta(T_0(x),x)| \to 0$ and back to the second step we infer $\|u_n - u\|_{\L^1(\Omega_\lambda)} \to 0$. 
	
	Sixth step: in the case of general $\BV$ boundary data we approximate $u_0$ w.r.t. $\|.\|_{\L^{\infty}(\bd \Omega)}$ by a sequence $(u_0^m)_{m \in \N}$ of smooth functions.
	The difference of the solutions to 
	\begin{align*}
		\SP< {c(x)}, {Du^m} > &= f(x) \cdot \Lm^2 \;, & u^m|_{\bd \Omega} &= u_0^m \;, &  \SP< {c(x)}, {T(x)} > &\geq \beta \cdot |\nabla T(x)| \;,\\
		\SP<{c(x)} ,{Du}> &= f(x) \cdot \Lm^2 \;, & u|_{\bd \Omega} &= u_0 \;, &  \SP< {c(x)}, {T(x)} > &\geq \beta \cdot |\nabla T(x)| \;,
	\end{align*}
	is exactly $\|u - u^m\|_{\L^{\infty}(\Omega)} = \|u_0 - u_0^m\|_{\L^{\infty}(\bd \Omega)}$. The same way, replacing $c$ with $c_n$ and $T$ with $T_n$, we obtain 
	$\|u_n - u_n^m\|_{\L^{\infty}(\Omega)} = \|u_0 - u_0^m\|_{\L^{\infty}(\bd \Omega)}$. 
	After having fixed an $m$ such big that $2\|u_0 - u_0^m\|_{\L^{\infty}(\bd \Omega)} \cdot \Lm^2(\Omega) \leq \varepsilon$, we decompose
	\[
		\|u-u_n\|_{\L^1(\Omega)} \leq 2\|u_0 - u_0^m\|_{\L^{\infty}(\bd \Omega)} \cdot \Lm^2(\Omega) + \|u^m-u^m_n\|_{\L^1(\Omega)} \leq \varepsilon + \|u^m-u^m_n\|_{\L^1(\Omega)}
	\]
	and apply the previous steps one up to five to estimate remainder $\|u^m-u^m_n\|_{\L^1(\Omega)}$.
	
	Final step: all the functions $u$, $u_n$ are elements of $\BV(\Omega)$. By the uniform bounds on the derivatives $\|Dc_n\|_{\L^1(\Omega)} \leq M_1$, $\|DN_n\|_{\L^1(\Omega)} \leq M_2$,
	and the boundedness of the sequence $m_{n,0}$,	the sequence of total variations $|Du_n|(\Omega)$ is bounded. 
	This is a consequence of theorem \ref{Theo:ElemBV}, where the bounds $\|Dc_n\|_{\L^1(\Omega)} \leq M_1$, $\|DN_n\|_{\L^1(\Omega)} \leq M_2$
	are necessary in the case of non-trivial right-hand sides $f$.
	By \cite[proposition 3.13]{AmbrosioBV}, the boundedness of the sequence $|Du_n|(\Omega)$and the $\L^1$ convergence $u_n \to u$ together imply the $\BV$ weak* convergence of $u_n$ to $u$. \hfill		
\end{proof}
\medskip

The well-posedness of the linear problem \refEq{eqn:LinProblem} is complete: we have shown the existence of a unique solution in $\BV(\Omega)$ and its continuous dependence on all of the data.
Now, we discuss the quasi-linear problem.

\section{The Quasi-Linear Problem}\label{Sect:QuasiLin}
The subject of this section is to proof the existence of a solution to the quasi-linear problem 
\begin{equation}\label{eqn:QuasiLinProblem}
	\begin{aligned}
		 \SP< {c[u](x)}, {Du} > &= f[u](x) \cdot \Lm^2 \quad \mbox{ in } \quad \Omega \wo \Sigma  \;, \\
		 u|_{\bd \Omega} &= u_0   \;, \\  
		 \SP< {c[u](x)}, {\nabla T(x)} > & \geq \beta \cdot |\nabla T(x)|   \;,
	\end{aligned}
\end{equation}
in $\BV(\Omega)$. Here, we allow the dependencies of $c$ and $f$ on the solution $u$ to be of a general functional type.

In the introduction we have pointed out, that we achieve this goal by fixed point theory applied to the operator $U[v]$, which denotes the solution operator to the family of linear problems
\begin{equation}\label{eqn:LinOfQuasiLin}
	\begin{aligned}
		 \SP< {c[v](x)}, {Du} > &= f[v](x) \cdot \Lm^2 \quad \mbox{ in } \quad \Omega \wo \Sigma  \;, \quad v \in \FX \subset \BV(\Omega)\\
		 u|_{\bd \Omega} &= u_0   \;, \\  
		 \SP< {c[v](x)}, {\nabla T(x)} > & \geq \beta \cdot |\nabla T(x)|   \;.
	\end{aligned}
\end{equation}
To be able to pursue this strategy, we have to extend the list of requirements on the coefficients $c$ and $f$ by assumptions concerning the functional argument $v$.
\medskip

\begin{req}\label{Req:TransportFieldFunc}
	(Transport fields)
	Regarding the functional argument, transport fields are maps of type
	\begin{align*}
		c: \L^1(\Omega) & \to \C^1(\Omega \wo \Sigma)^2 \;, & \mbox{with} & & c[\dotarg](x): \L^1(\Omega) &\to \R^2 \;,
	\end{align*}
	and we assume them to satisfy:
	\smallskip
	
	\begin{enumerate}[1.]
		\item For fixed $v \in \L^1(\Omega)$ the function $c[v]: \Omega \wo \Sigma \to \R^2 $ is a transport field according to requirement \ref{Req:TransportField}.
			\smallskip
		\item Uniformity of the unit speed and causality condition:
			\smallskip
			\begin{enumerate}[a)]
				\item $|c[v](x)| = 1$ for all $x \in \Omega \wo \Sigma$ and for all $v \in \L^1(\Omega)$.
					\smallskip
				\item There is a uniform lower bound $\beta > 0$ such that
					\begin{align*}
						\beta \leq \SP< {c[v](x)}, {N(x)} >  &\leq 1 \qquad \forall \; x \in \cl{\Omega} \wo \Sigma 
						& \mbox{and}\quad &\; \forall \;  v \in \L^1(\Omega) \;.
					\end{align*}
				\item Both conditions hold for the one-sided limits of $c[v]$ on the relatively open $\C^1$ arcs $\inner{\Sigma}_k$ of $\Sigma$.
			\end{enumerate}
			\smallskip
		\item Bounds and continuity:
			\smallskip
			\begin{enumerate}[a)]
				\item The map $D_x c: \L^1(\Omega) \to \C(\Omega \wo \Sigma)^{2 \times 2}$ -- the derivative of $c[v]$ w.r.t. the variable $x$ -- 
					is $\L^1$ bounded by 
					\[
						\|D_x c[v]\|_{L^1(\Omega)} < M_1 \qquad \forall \; v \in \L^1(\Omega) \;.
					\] 
				\item $c$ is continuous in the following manner:  if $v \in \L^1(\Omega)$ and $(v_n)_{n \in \N}$ is a sequence in $\L^1(\Omega)$ with $\|v-v_n\|_{L^1(\Omega)} \to 0$, 
					then the sequence $(c[v_n])_{n \in \N}$ converges uniformly to $c[v]$, i.e., $\|c[v]-c[v_n]\|_{\infty} \to 0$.
			\end{enumerate}
	\end{enumerate}
\end{req}
\bigskip

\begin{req}\label{Req:RHSFunc}
	(Right-hand sides)
	Regarding the functional argument, right-hand sides are maps of type
	\begin{align*}
		f: \L^1(\Omega) & \to \C^1(\cl{\Omega}) \;, & \mbox{with} & & f[\dotarg](x): \L^1(\Omega) &\to \R \;,
	\end{align*}
	and we assume them to satisfy:
	\smallskip
	\begin{enumerate}[a)]
		\item The map $f$ is bounded by 
			\[
				 \|f[v]\|_{\infty} \leq M_2 \qquad \forall \; v \in \L^1(\Omega) \;.
			\] 
		\item The map $\nabla_x f: \L^1(\Omega) \to \C(\cl{\Omega})^2$ -- the derivative of $f[v]$ w.r.t. the variable $x$ -- is bounded by
			\[
				 \|\nabla_x f[v]\|_{\infty} \leq M_3 \qquad \forall \; v \in \L^1(\Omega) \;.
			\]
		\item $f$ is continuous in the following manner:  if $v \in \L^1(\Omega)$ and $(v_n)_{n \in \N}$ is a sequence in $\L^1(\Omega)$ with $\|v-v_n\|_{L^1(\Omega)} \to 0$, 
			then the sequence $(f[v_n])_{n \in \N}$ converges uniformly to $f[v]$, i.e., $\|f[v]-f[v_n]\|_{\infty} \to 0$.
	\end{enumerate}
\end{req}
\medskip

In addition, we define the subsets of $\BV(\bd \Omega)$ and $\BV(\Omega)$ with which we will work later on.
\medskip

\begin{defi}\label{Def:FuncDomains}
	Let $M_1$, $M_2$, $M_3$ be the bounds from the requirements stated above. 
	\begin{enumerate}[a)]
		\item We denote by
			\[
				\FB = \FB(\bd \Omega) := \{v \in \BV(\bd \Omega) : \|v\|_{\L^{\infty}(\bd \Omega)} \leq M_4 \;,\; |Dv|(\bd \Omega) \leq M_5\}
			\]
			the set of boundary functions.
		\item Let the constants $M_{*} \in \R$ and $M_{**} \in \R$ be given by
			\begin{equation}\label{eqn:SelfMapBound}
			\begin{aligned}
				M_{*}  := &\left(M_4 + \frac{M_2}{\beta \cdot m_0}\right) \cdot \Lm^2(\Omega) \;, \\
				M_{**} := & \; 2 \cdot \left(M_4 + \frac{M_2}{\beta \cdot m_0}\right) \cdot \Hm^1(\Sigma) + \frac{M_5}{\beta \cdot m_0}  
				+ \left( \frac{M_2}{\beta} + \frac{M_3}{\beta^2 \cdot m_0}\right) \cdot \Lm^2(\Omega) \\
				& + \frac{M_2}{\beta^3 \cdot m_0} \cdot \left( M_1 + \|DN\|_{L^1(\Omega)}\right) \;.
			\end{aligned}
			\end{equation}
			We set
			\[
				\FX = \FX(\Omega) := \{ v \in \BV(\Omega) : \|v\|_{L^1(\Omega)} \leq M_{*} \; ,\; |Dv|(\Omega) \leq M_{**}\} \;.
			\]
	\end{enumerate}
\end{defi}
\medskip

Our aim is to view problem \refEq{eqn:QuasiLinProblem} as fixed point problem. The next corollary
justifies the change of our viewpoint.
\medskip

\begin{cor} \label{Cor:OpDefAndSelfMap}
	For fixed $v \in \L^1(\Omega)$, problem \refEq{eqn:LinOfQuasiLin}
	meets the requirements of the linear problem \refEq{eqn:LinProblem} and			
	its unique solution, which we denote by $U[v]$, defines an operator $U: \L^1(\Omega) \to \FX$.
	
	Moreover, after restriction to $\FX$, this operator is a self-mapping $U:\FX \to \FX$ .
\end{cor}
\medskip

\begin{proof} 
	Let $v \in \L^1(\Omega)$ be arbitrary. By requirement \ref{Req:TransportFieldFunc} part 1 the field $c[v]$ is a transport field according to requirement \ref{Req:TransportField},
	while $f[v] \in \C^1(\cl{\Omega})$. Thus, the requirements of the linear problem \refEq{eqn:LinProblem} are satisfied and the linear theory
	ensures the existence of a unique solution $U[v] \in \BV(\Omega)$. Hence, the map $U: \L^1(\Omega) \to \BV(\Omega)$ is well-defined.
	
	By theorem \ref{Theo:ElemBV} we get the following estimates for fixed $v$: 
	\begin{align*}
		\|U[v]\|_{\L^{\infty}(\Omega)} \leq &  \|u_0\|_{\L^{\infty}(\bd \Omega)} +  \frac{ \|f[v]\|_{\infty}}{\beta \cdot m_0} \;,\\
		|DU[v]|(\Omega) \leq & \;  2 \cdot \|U[v]\|_{\L^{\infty}(\Omega)} \cdot \Hm^1(\Sigma) + \frac{|Du_0|(\bd \Omega)}{\beta \cdot m_0}  
		+ \left( \frac{\|f[v]\|_{\infty}}{\beta} + \frac{\|\nabla f[v] \|_{\infty}}{\beta^2 \cdot m_0}\right) \cdot \\ 
		& \Lm^2(\Omega) + \frac{\|f[v]\|_{\infty}}{\beta^3 \cdot m_0} \cdot \left( \|Dc[v]\|_{L^1(\Omega)} + \|DN\|_{L^1(\Omega)}\right) \;.
	\end{align*}
	Plugging into these estimates the uniform bounds $M_1$, $M_2$, $M_3$  from requirements \ref{Req:TransportFieldFunc} and \ref{Req:RHSFunc}  and the bounds $M_4$, $M_5$ on $u_0 \in \FB$,
	it is easy to see that we obtain the uniform upper bounds
	\begin{align*}
		\|U[v]\|_{\L^1(\Omega)} &\leq M_{*}  \;, & |DU[v]|(\Omega)& \leq M_{**} \;.
	\end{align*}
	Summarizing, the operator $U$ is in fact of type $	U:\L^1(\Omega) \to \FX$.
	
	Because of $\FX \subset \L^1(\Omega)$ we can restrict the domain of $U$ to $\FX$, and obtain the self-mapping $U:\FX\to \FX$. \hfill
\end{proof}
\medskip

The final ingredient that we need is the continuity of the operator $U$.
\medskip 

\begin{cor}\label{Lem:OpCont}(Continuity of $U$)
	The operator $U:\FX \to \FX$ from corollary \ref{Cor:OpDefAndSelfMap} b) is sequentially continuous w.r.t. the $\BV$ weak* topology.
\end{cor}
\medskip

\begin{proof}
	Let $(v_n)_{n \in \N}$ be a sequence in $\FX$ which tends to $v \in \FX$ w.r.t. the $\BV$ weak* topology. Then, we have in particular $\|v-v_n\|_{L^1(\Omega)} \to 0$.
	According to requirement \ref{Req:TransportFieldFunc}, the sequence $c_n := c[v_n]$ (with limit $c := c[v]$) and the constant sequence $T_n = T$ of time functions satisfy the requirements
	of \ref{Theo:ContDepend}. So as in section \ref{Sect:Stab}, we conclude $\|U[v_n] - U[v]\| \to 0$, and the weak* convergence $U[v_n] \weaksto U[v] $
	is a consequence of the boundedness of $ |DU[v]|(\Omega) \leq M_{**}$. \hfill
\end{proof}
\medskip

With the corollaries \ref{Cor:OpDefAndSelfMap} and \ref{Lem:OpCont}, we have sufficient conditions to conclude the existence of a solution to problem \refEq{eqn:QuasiLinProblem}.
By corollary \ref{Cor:OpDefAndSelfMap}, the quasi-linear problem is equivalent to the fixed point problem $u = U[u]$. 
The set $\FX$ from definition \ref{Def:FuncDomains}, is non-empty, convex and, by \cite[theorem 3.23]{AmbrosioBV} sequentially compact w.r.t. the $\BV$ weak* topology.
Finally, corollary \ref{Lem:OpCont} yields the sequential continuity of the map $U:\FX \to \FX$.
Hence, the existence of a fixed point $u = U[u]$ is the consequence of the Schauder-Tychonoff fixed point theorem (see \cite[chapter 9.3]{Zeidler:FPT}). 

So far we have shown the existence of a fixed point, but uniqueness cannot be expected as the following example illustrates.
\medskip

\paragraph{Example of non-uniqueness}
As before let $\Omega \wo \Sigma$ be the domain of the PDE and $T$ a time function. 
We consider an almost linear problem, where the transport field does not depend on $u$, with zero boundary condition
\begin{equation} \label{eqn:AlmostLin}
	\begin{aligned}
		 \SP< {c(x)}, {Du} > &= f[u](x) \cdot \Lm^2   \;,\quad & u|_{\bd \Omega} &= 0 \;,\quad \SP<{c(x)},{\nabla T(x)}> \geq \beta \cdot |\nabla T(x)|\;.
	\end{aligned}
\end{equation}
Furthermore, we choose a right-hand side of the form
\begin{equation} 
	f[v](x) = g\left(\, \|v\|_{\L^1(\Omega)} \,\right) \;,
\end{equation}
which is independent of $x$. Thereby $g:\R \to \R$ is a continuous bounded function so that $f$ satisfies requirements \ref{Req:RHSFunc}.

After having fixed the functional argument of $f$ in equation \refEq{eqn:AlmostLin}, the solution $U[v]$ is given by
\begin{equation}\label{eqn:SolAlmostLin}
	U[v](x) = \int\limits_0^{T_0(x)} f_0[v] \circ \eta(\tau ,x) \: d\tau \;, 
\end{equation}
according to equation \refEq{eqn:SolOrig}. Because the backward characteristic solves $\eta'(\dotarg,x) = -c_0 \circ \eta(\dotarg,x)$ and $c$ is normed, we have
\[
	|\eta'(\dotarg,x)| = |c_0 \circ \eta(\dotarg,x)| = \left| \frac{c}{\SP<{c},{\nabla T_0}>} \circ \eta(\dotarg,x) \right| = \frac{1}{\SP<{ c},{\nabla T_0} >} \circ \eta(\dotarg,x) \;.
\]
And because $f[v]$ does not depend on $x$, the integrand in equation \refEq{eqn:SolAlmostLin} reduces to
\[  
	f_0[v] \circ \eta(\dotarg,x) = f[v] \cdot \frac{1}{\SP<{ c},{\nabla T_0} >} \circ \eta(\dotarg,x) =  f[v] \cdot |\eta'(\dotarg,x) | \;.
\]
Thus, $U[v](x) = f[v] \cdot a(x)= g\left(\, \|v\|_{\L^1(\Omega)} \,\right) \cdot a(x)$ has separated variables and $a(x)$ is the arc length of the characteristic
$\eta(\dotarg,x)$ which connects $x=\eta(0,x)$ and the point $\eta(T_0(x),x)$ on the boundary $\bd \Omega$.

If now $u$ is a fixed point of $U$, it must satisfy
\begin{equation} \label{eqn:FPcond}
	u(x) =  g\left(\, \|u\|_{\L^1(\Omega)} \,\right) \cdot a(x)\;.
\end{equation}
Consequently, $u(x) = \alpha \cdot a(x)$ is a scalar multiple of the arc length function $a$. Plugging this into condition \refEq{eqn:FPcond}, we obtain by
\[
	\alpha \cdot a(x) = g\left( \alpha \cdot \|a\|_{\L^1(\Omega)} \right) \cdot a(x) \quad \Rightarrow \quad \alpha = g\left( \alpha \cdot \|a\|_{\L^1(\Omega)} \right) =: \tilde{g}(\alpha) 
\]
a scalar fixed point problem for $\tilde{g}$. Finally, if $\tilde{g}$ is a continuous bounded function with many fixed points $\alpha$, e.g.
\[
	\tilde{g}(t) =
	\begin{cases}
		-1 &, t \leq -1 \\
		t &, -1 < t \leq 1 \\
		1 &, 1 < t
	\end{cases} \;,
	\quad \mbox{with} \quad
	g(t) = \tilde{g}\left( \, \frac{t}{\|a\|_{\L^1(\Omega)}} \, \right),
\]
then the operator $U$ has as many fixed points as $\tilde{g}$.

\section{Conclusion}
We have shown in section \ref{Sect:LPExist} that the linear hyperbolic Dirichlet problem --
with interior outflow set $\Sigma$ and mere inflow boundary --
admits a solution in $\BV(\Omega)$ under the assumptions of section \ref{Sect:Problem}.
The crucial ingredient here was the knowledge of a time function $T$ w.r.t. which the transport field $c$ is causal.
Furthermore, we have shown in section \ref{Sect:Stab} that this solution is unique and depends continuously on all the data of the problem.

Certainly, one might ask why we prefer the space $\BV(\Omega)$, because, if the boundary data $u_0$ were $\C^1$, 
one could solve in $\C^1(\Omega \wo \Sigma)$ (this possibility is contained in our approach).
The advantage of working in $\BV(\Omega)$ is three-fold:
\begin{enumerate}[1.]
\item Our solution solves the problem in $\Omega \wo \Sigma$, but is in fact an element of $\BV(\Omega)$ according to theorem \ref{Theo:ElemBV} part b).
	Therewith we get a description of what happens to $u$ on $\Sigma$ in the language of geometric measure theory and can in this sense close the gap $\Sigma$.
\item The benefit of closing the gap is that the domain  $\FX \subset \BV(\Omega)$ of the operator $U$ from section \ref{Sect:QuasiLin} is compact w.r.t. the $\BV$ weak* topology.  
	This fact together with the continuous dependence allowed us to conclude the existence of a solution to the quasi-linear case by employing the Schauder fixed point theorem.
\item A further advantage is that we can allow for more general time functions $T$ and stop sets $\Sigma$. For example we can use a $T$ with a saddle node as shown in figure \ref{Fig:Saddle5}.
\begin{figure}%[H]
	\begin{center}
		\begin{tikzpicture}[scale=7]
			\colorlet{darkgreen}{green!80!black}
			
			% background
			\fill[gray] (0,0.15) rectangle (1,0.85);
			
			% foreground circles
			\filldraw[white] (0.375,0.5) circle (0.25cm);
			\filldraw[white] (0.625,0.5) circle (0.25cm);
			
			% boundary
			\draw[blue,thick, rounded corners=2pt] (0.875,0.5) arc (0:120:0.25cm)  
			arc (60:300:0.25cm) arc (240:360:0.25cm);
			
			% saddle level
			\filldraw[fill=white,draw=darkgreen,thick] (0.375,0.5) circle (0.125cm);
			\filldraw[fill=white,draw=darkgreen,thick] (0.625,0.5) ellipse (0.125cm and 0.1cm);
			\filldraw[darkgreen] (0.5,0.5) circle (0.075mm);
							
			% max left
			\filldraw[red] (0.375,0.5) circle (0.075mm);
			
			% max right
			\filldraw[red] (0.625,0.5) circle (0.075mm);
		\end{tikzpicture}
	\end{center}
	\caption{Three levels of a time function $T$ with a saddle node in the middle. The white area is the domain $\Omega$ with its boundary, which is the start level $T=0$, in blue.
		 The green line is the saddle level of $T$ and the green center dot is the saddle node of $T$. The red dots are the maximal level of $T$, i.e.,
		 the final stop set which is disjoint.}  
	\label{Fig:Saddle5}
\end{figure}
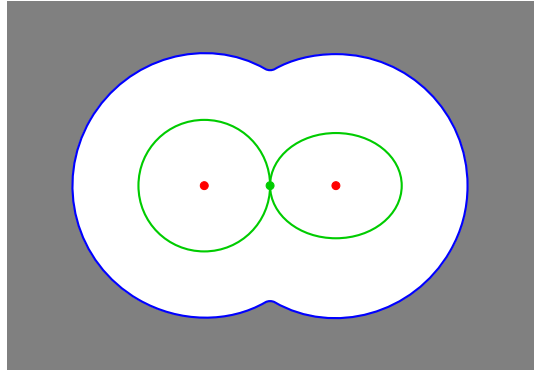
Here, in a first step, one solves from the boundary to the saddle level. The saddle node in the middle is the first stop set so to speak.
For the second step, we consider two decoupled problems on the remaining parts.
At the saddle node we will in general get a jump discontinuity even if the data on $\bd \Omega$ is smooth.
That means the problems on the remaining parts -- (re-)started on the saddle level -- are supplied with $\BV$-data. 
For more on this topic see \cite[chapter 5]{MyDiss}.
\smallskip
\end{enumerate}

In this paper we have restricted the discussion to two-dimensional domains and in the proofs of theorems \ref{Theo:ElemBV} and \ref{Theo:ContDepend}
for arguments concerning the boundary data we used strongly the fact that $\BV$ functions of one variable are essentially bounded.
In order to generalize to $d$-dimen\-sional domains, $d > 2$, with a $(d-1)$-dimensional boundary, one must require the boundary data to be $\BV(\bd \Omega) \cap \L^{\infty}(\bd \Omega)$.

\section{Acknowledgements}
The author would like to thank Folkmar Bornemann and Nick Trefethen for their advice and the inspiring discussions.

\bibliographystyle{siam}
\bibliography{DPaper1}

\begin{thebibliography}{10}

\bibitem{Amann:1990:ODE}
{\sc H.~Amann}, {\em Ordinary Differential Equations: An Introduction to
  Nonlinear Analysis}, de Gruyter, Berlin, 1990.

\bibitem{AmbrosioBV}
{\sc L.~Ambrosio, N.~Fusco, and D.~Pallara}, {\em Functions of Bounded
  Variation and Free Discontinuity Problems}, Oxford University Press, 2000.

\bibitem{FBTM07}
{\sc F.~Bornemann and T.~M{\"a}rz}, {\em Fast image inpainting based on
  coherence transport}, J. Math. Imaging Vision, 28 (2007), pp.~259--278.

\bibitem{ConwaySmoller66}
{\sc E.~Conway and J.~Smoller}, {\em Global solutions of the cauchy-problem for
  quasi-linear first-order equations in several space variables}, Comm. Pure
  Appl. Math., 19 (1966), pp.~95--105.

\bibitem{Evans:PDE}
{\sc L.~Evans}, {\em Partial Differential Equations}, vol.~19 of Graduate
  Studies in Mathematics, American Mathematical Society, 1998.

\bibitem{MyDiss}
{\sc T.~M{\"a}rz}, {\em First Order Quasi-Linear PDEs with BV Boundary Data and
  Applications to Image Inpainting}, PhD thesis, Technische Universit{\"a}t
  M{\"u}nchen,
  \url{http://nbn-resolving.de/urn/resolver.pl?urn:nbn:de:bvb:91-diss-20100429%
-977864-1-6}, 2010.

\bibitem{Telea04}
{\sc A.~Telea}, {\em An image inpainting technique based on the fast marching
  method}, J. Graphics Tools, 9 (2004), pp.~23--34.

\bibitem{Volpert67}
{\sc A.~Vol'pert}, {\em Spaces bv and quasilinear equations}, Math. USSR Sb.,
  17 (1967), pp.~225--267.

\bibitem{Walter:DiffInEq}
{\sc W.~Walter}, {\em Differential and Integral Inequalities}, Springer,
  Berlin, 1970.

\bibitem{Zauderer:PDE}
{\sc E.~Zauderer}, {\em Partial Differential Equations of Applied Mathematics},
  John Wiley and Sons, New York, second~ed., 1989.

\bibitem{Zeidler:FPT}
{\sc E.~Zeidler}, {\em Fixed-Point Theorems}, vol.~1 of Nonlinear Functional
  Analysis and its Applications, Springer Verlag, 1993.

\end{thebibliography}

\end{document}